%% file: main.tex
\documentclass[12pt,a4paper,leqno]{article}

\usepackage[hypertexnames=false]{hyperref}

\usepackage[centertags]{amsmath}
\usepackage{amsthm}
\usepackage{amssymb,exscale}
\usepackage{paralist}
\usepackage{latexsym}
\usepackage{graphicx}

\input{preamble}

\begin{document}

\title{Linear response and moderate deviations:\\ hierarchical
  approach. III}

\author{Boris Tsirelson}

\date{}
\maketitle

\begin{abstract}
The Moderate Deviations Principle (MDP) is well-understood for sums of
independent random variables, worse understood for stationary random
sequences, and scantily understood for random fields.
Here it is established for splittable random fields.
\end{abstract}

\setcounter{tocdepth}{2}
\tableofcontents

\numberwithin{equation}{section}

\section[Introduction, and main result formulated]
  {\raggedright Introduction, and main result formulated}
\label{sect1}
\input{sect1}

\section[Single-step bounds]
   {\raggedright Single-step bounds}
\label{sect2}
\input{sect2}

\section[Multi-step bounds based on Lemma \ref{2.2}]
  {\raggedright Multi-step bounds based on Lemma \ref{2.2}}
\label{sect3}
\input{sect3}

\section[Multi-step bounds based on Lemma \ref{2.3}]
  {Multi-step bounds based on Lemma \ref{2.3}}
\label{sect4}
\input{sect4}

\section[Quadratic approximation based on Sect.~\ref{sect3}]
  {Quadratic approximation based on Sect.~\ref{sect3}}
\label{sect5}
\input{sect5}

\section[Quadratic approximation based on Sect.~\ref{sect4}]
  {Quadratic approximation based on Sect.~\ref{sect4}}
\label{sect6}
\input{sect6}

\bigskip
\filbreak
{
\small
\begin{sc}
\parindent=0pt\baselineskip=12pt
\parbox{4in}{
Boris Tsirelson\\
School of Mathematics\\
Tel Aviv University\\
Tel Aviv 69978, Israel
\smallskip
\par\quad\href{mailto:tsirel@post.tau.ac.il}{\tt
 mailto:tsirel@post.tau.ac.il}
\par\quad\href{http://www.tau.ac.il/~tsirel/}{\tt
 http://www.tau.ac.il/\textasciitilde tsirel/}
}

\end{sc}
}
\filbreak

\end{document}

%% file: preamble.tex
\swapnumbers
\theoremstyle{definition}
\newtheorem{theorem}[equation]{Theorem}
\newtheorem{lemma}[equation]{Lemma}
\newtheorem{corollary}[equation]{Corollary}

\newtheorem{proposition}[equation]{Proposition}
\newtheorem{remark}[equation]{Remark}

\newtheorem{innercustomprop}{Proposition}
\newenvironment{customprop}[1]
  {\renewcommand\theinnercustomprop{#1}\innercustomprop}
  {\endinnercustomprop}

\renewcommand{\phi}{\varphi}

\newcommand{\D}{\mathrm{d}}
\newcommand{\E}{\mathrm{e}}

\newcommand{\ti}{\tilde}

\renewcommand{\(}{\bigl(}
\renewcommand{\)}{\bigr)\vphantom{)}}

\newcommand{\impl}{\;\Longrightarrow\;}
\newcommand{\imp}{$ \Longrightarrow $ }

\newcommand{\vol}{\operatorname{vol}}
\newcommand{\length}{\operatorname{length}}
\newcommand{\arglength}{\operatorname{arglength}}
\newcommand{\width}{\operatorname{width}}

\newcommand{\const}{\operatorname{const}}

\newcommand{\eps}{\varepsilon}

\newcommand{\si}{\sigma}

\newcommand{\de}{\delta}
\newcommand{\De}{\Delta}
\newcommand{\al}{\alpha}
\newcommand{\be}{\beta}

\newcommand{\la}{\lambda}

\newcommand{\Ex}{\mathbb E\,}
\newcommand{\R}{\mathbb R}

\newcommand{\PR}[1]{\mathbb{P}\mskip1.5mu\bigg(\mskip1.5mu#1\mskip1.5mu\bigg)}

\newcommand{\myatop}[2]{{\genfrac{}{}{0pt}{}{#1}{#2}}}

\newcommand{\close}[1]{$#1$\nobreakdash-\hspace{0pt}close}

%% file: sect1.tex
Splittable random fields defined in \cite{II} generalize splittable
random processes defined in \cite{I}. The Moderate Deviations Principle
(MDP) for splittable processes is obtained in \cite{I} via ``Linear
Response Principle'' (LRP), the latter being the quadratic logarithmic
asymptotics for exponential moments. Both LRP and MDP are generalized
here for splittable centered measurable stationary (``CMS'') random
fields, as defined in \cite[Sect.~1]{II} (see \cite[(1.1)--(1.3)]{II}
for CMS random fields, and \cite[Def.~1.4]{II} for splittability).

\begin{theorem}[\emph{``linear response''}]\label{theorem1}
The following limit exists for every splittable CMS random
field $ X $ on $ \R^d $:
\[
\lim_\myatop{ r_1,\dots,r_d\to\infty, \la\to0 }{ \la\log^d (r_1\dots
r_d) \to 0 }
\frac1{ r_1\dots r_d\la^2
} \log \Ex \exp \la \int_{[0,r_1]\times\dots\times[0,r_d]} X_t \, \D
t \, .
\]
\end{theorem}

That is, for every $ \eps $ there exist $ R $ and $ \de $ such that the
given expression is \close{\eps} to the limit for all $
r_1,\dots,r_d \in [R,\infty) $ and all $ \la \ne 0 $ such that $
|\la| \log^d (r_1\dots r_d) \le \de $.\,\footnote{%
 Of course, $ \log^n x $ means $ (\log x)^n $.}

We denote this limit by $ \si^2/2 $, $ \si \in [0,\infty) $.

\begin{corollary}[\emph{moderate deviations}]\label{1.2}
Let $ X $ and $ \si $ be as above, and $ \si \ne 0 $. Then
\[
\lim_{\textstyle\myatop
 { r_1,\dots,r_d\to\infty, c\to\infty }
 { \left( c\log^d (r_1,\dots,r_d) \right)^2 / (r_1\dots r_d) \to 0 }
}
\frac1{c^2} \log \PR{ \int_{[0,r_1]\times\dots\times[0,r_d]} X_t \, \D
t \ge c\si \sqrt{r_1\dots r_d} } = -\frac12 \, .
\]
\end{corollary}

\begin{corollary}\label{1.3}
The distribution of $ (r_1\dots
r_d)^{-1/2} \int_{[0,r_1]\times\dots\times[0,r_d]} X_t \, \D t $
converges (as $ r_1,\dots,r_d \to \infty $) to the normal distribution $
N(0,\si^2) $.
\end{corollary}

For the proofs see Sect.~\ref{sect6}.

%% file: sect2.tex
\emph{A digression on boxes.}
By a box we mean a set of the form $ B = [0,r_1] \times \dots \times
[0,r_d] \subset \R^d $ where $ r_1,\dots,r_d \in (0,\infty) $. We
denote
\[
\vol B = r_1 \dots r_d \, ; \quad 
\length B = \max ( r_1, \dots, r_d ) \, ; \quad
\width B = \min ( r_1, \dots, r_d ) \, ;
\]
clearly, $ (\width B)^d \le \vol B \le (\length B)^d $. We consider
the longest side of $B$,
\[
\arglength B = \min \{ k : r_k = \length B \}
\]
and halve it, getting another box denoted by $ B / 2 $:
\[
B/2 = [0,s_1] \times\dots\times [0,s_d] \, , \quad \text{where} \quad
  s_k = \begin{cases}
    \frac12 r_k &\text{if } k=\arglength B , \\
    r_k &\text{otherwise} .
    \end{cases}
\]
Note the equality
\begin{equation}\label{2.05}
\width(B/2) = \min \Big( \width B, \frac12 \length B \Big) \, ;
\end{equation}
it will be used in the next section.

\emph{Back to random fields.}
From now on we assume that a splittable CMS random field $X$ on $ \R^d $
is given, and denote
\[
f_B(\la) = \log \Ex \exp \frac{\la}{\sqrt{\vol B}} \int_B X_t \, \D t
\]
for arbitrary box $ B \subset \R^d $ and $ \la \in \R $; note that $
f_B(\la) \in [0,\infty] $ (since
$ \Ex \exp \frac{\la}{\sqrt{\dots}} \int \dots \ge \Ex \(
1+\frac{\la}{\sqrt{\dots}} \int \dots \) = 1 $). Also, we denote for
convenience
\[
R(v) = v^{\frac1d} \quad \text{and} \quad S(v) =
v^{\frac{d-1}d} \quad \text{for } v \in (0,\infty) \, .
\]
Sometimes the case $ d=1 $ needs special attention. By convention, $
x^0 = 1 $ for all $ x \in \R $ (not only for $ x>0 $).

\begin{lemma}\label{2.1}
There exists $ C_1 \in (1,\infty) $ such that for every box $B$
satisfying $ \width B \ge C_1 $ hold

(a) $\displaystyle f_B (\la) \le \frac2p f_{B/2} \Big( \frac{p\la}{\sqrt2} \Big) +
C_1 \frac{p}{p-1} \cdot \frac{\la^2}{R(\vol B)} $\\
whenever $ C_1 |\la| \le \frac{p-1}p \sqrt{\vol B} \log^{-(d-1)}
S(\vol B) $;

(b) $\displaystyle f_B (\la) \ge 2p f_{B/2} \bigg( \frac{\la}{p\sqrt2} \bigg) -
C_1 \frac{1}{p-1} \frac{\la^2}{R(\vol B)} $\\
whenever $ C_1 |\la| \le (p-1) \sqrt{\vol B} \log^{-(d-1)}
S(\vol B) $.
\end{lemma}

(Here ``whenever\ \dots'' means ``for all $ p \in (1,\infty) $ and
$ \la \in \R $ satisfying\ \dots''.)

\begin{proof}
By \cite[Lemma 2.4]{II} (for the constant $C_1$ given
by \cite[(1.6)]{II} and a single index $i$),
\[
f_B (\la) \le \frac2p f_{B/2} \Big( \frac{p\la}{\sqrt2} \Big) +
C_1 \frac{p}{p-1} \cdot \frac{\la^2}{\length B}
\]
whenever $ C_1 |\la| \le \frac{p-1}p \sqrt{\vol
B} \log^{-(d-1)} \frac{\vol B}{\length B} $; inequality (a) follows,
since $ \length B \ge R(\vol B) $ and $ \frac{\vol B}{\length B} \le
S(\vol B) $.

Inequality (b) follows similarly from \cite[Lemma 2.10]{II}, but some
clarifications are needed. There, the second inequality for $ r=s $
says (in our terms) that
\[
f_B (\la) \ge
2p f_{B/2} \bigg( \frac{\la}{p\sqrt2} \bigg) -
(p-1) g_{B_0} \bigg( \frac{1}{p-1} \frac{\la}{\sqrt{\length
B}} \bigg) \, ,
\]
where (see \cite[(1.6)]{II})
\[
g_{B_0} (\mu) \le C_1 \mu^2 \quad \text{for }
C_1 |\mu| \le \frac{ \sqrt{\vol B_0} }{ \log^{d-1} \vol B_0 } \, ;
\]
also $ \vol B_0 = \frac{\vol B}{\length B} \le S(\vol B) $ (and
$ \width B_0 = \width B \ge C_1 $).
\end{proof}

From now on $C_1$ is the constant given by Lemma \ref{2.1}; it depends
on the random field $X$ only. Every larger constant is also a
legitimate $ C_1 $; in Sect.~\ref{sect3} we'll assume that $ C_1 \ge 3
$.

\begin{lemma}\label{2.2}
Let a box $B$ satisfy $ \width B \ge C_1 $; denote $ x
= \sqrt{\frac{C_1}{R(\vol B)}} $.

(a) If numbers $ u,\de \in (0,\infty) $ satisfy
\[
f_{B/2}(\mu) \le u^2 \mu^2 \quad \text{for all } \mu \in [-\de,\de] \,
,
\]
then
\[
f_B(\la) \le (u+x)^2 \la^2 \quad \text{for all } \la \in [-\De,\De] \,
,
\]
where $ \De
= \min \( \frac{\sqrt2}p \de, \frac1{C_1} \frac{p-1}p \sqrt{\vol
B} \log^{-(d-1)} S(\vol B) \) $ and $ p = \frac{u+x}u $;

(b) if numbers $ u \in (x,\infty) $, $ \de \in (0,\infty) $ satisfy
\[
f_{B/2}(\mu) \ge u^2 \mu^2 \quad \text{for all } \mu \in [-\de,\de] \,
,
\]
then
\[
f_B(\la) \ge (u-x)^2 \la^2 \quad \text{for all } \la \in [-\De,\De] \,
,
\]
where $ \De
= \min \( p \de \sqrt2, \frac1{C_1} (p-1) \sqrt{\vol
B} \log^{-(d-1)} S(\vol B) \) $ and $ p = \frac u{u-x} $.
\end{lemma}

\begin{proof}
Item (a). Using Lemma \ref{2.1}(a) and taking into account that $ p>1
$, $ \big| \frac{p\la}{\sqrt2} \big| \le \de $ and $ C_1
|\la| \le \frac{p-1}p \sqrt{\vol B} \log^{-(d-1)} S(\vol B) $ we have
$ f_B(\la) \le \frac2p u^2 \( \frac{p\la}{\sqrt2} \)^2 + \frac{p}{p-1}
x^2 \la^2 = \( u^2 p + x^2 \frac{p}{p-1} \) \la^2 $, and $ u^2 p +
x^2 \frac{p}{p-1} = u^2 \frac{u+x}u + x^2 \frac{u+x}x = (u+x)^2 $.

Item (b). Using Lemma \ref{2.1}(b) and taking into account that $ p>1
$, $ \big| \frac{\la}{p\sqrt2} \big| \le \de $ and $ C_1
|\la| \le (p-1) \sqrt{\vol B} \log^{-(d-1)} S(\vol B) $ we have
$ f_B(\la) \ge 2p u^2 \( \frac{\la}{p\sqrt2} \)^2 - \frac1{p-1}
x^2 \la^2 = \( \frac{u^2}p -\frac{x^2}{p-1} \) \la^2 $, and
$ \frac{u^2}p -\frac{x^2}{p-1} = u^2 \frac{u-x}u - x^2 \frac{u-x}x =
(u-x)^2 $.
\end{proof}

\begin{lemma}\label{2.3}
Let a box $B$ satisfy $ \width B \ge C_1 $; denote
$ x = \frac1{R(\vol B) \log^{d-1} S(\vol B) } $,
$ y = \frac{C_1}{\sqrt{\vol(B/2)}} \log^{d-1} S(\vol B) $.
Let $ \la,\mu \in \R $ satisfy $ \la \mu > 0 $; denote
\[
\al = \frac{\sqrt2}{|\la|} - \frac1{|\mu|} \, , \quad \be = \frac{ f_B(\la) }{
|\la| \sqrt{\vol B} } - \frac{ f_{B/2}(\mu) }{ |\mu| \sqrt{\vol(B/2)}
} \, .
\]

\noindent (a) If $ \al \ge y $, then $ \be \le x $;

\noindent (b) if $ \al \le -y $, then $ \be \ge -x $.
\end{lemma}

\begin{proof} \let\qed\relax
Item (a). We take $ p = \frac{\mu\sqrt2}{\la} $, note that $ p>1 $
(since $ \al \ge y > 0 $), $ \mu = \frac{p\la}{\sqrt2} $, and
$ \frac{p-1}{p|\la|} = \frac1{|\la|} - \frac1{|\mu|\sqrt2}
= \frac{\al}{\sqrt2} \ge \frac{y}{\sqrt2} = \frac{C_1}{\sqrt{\vol
B}} \log^{d-1} S(\vol B) $, that is, $ C_1
|\la| \le \frac{p-1}p \sqrt{\vol B} \log^{-(d-1)} S(\vol B) $. Using
Lemma \ref{2.1}(a),
\begin{multline*}
\be = \frac{ f_B(\la) }{ |\la| \sqrt{\vol B} } - \frac{
 f_{B/2}\(\frac{p\la}{\sqrt2}\)
 }{ \frac{p|\la|}{\sqrt2} \sqrt{\frac{\vol B}{2} } } =
 \frac1{ |\la| \sqrt{\vol B} } \bigg( f_B(\la) - \frac2p
 f_{B/2} \Big(\frac{p\la}{\sqrt2}\Big) \bigg) \le \\
\le \frac1{ |\la| \sqrt{\vol B} }
 C_1 \frac{p}{p-1} \frac{\la^2}{R(\vol B)} = C_1
 |\la| \cdot \frac{p}{p-1} \cdot \frac1{R(\vol B)\sqrt{\vol B}} \le \\
\le \frac{p-1}p \sqrt{\vol B} \log^{-(d-1)} S(\vol
 B) \cdot \frac{p}{p-1} \cdot \frac1{R(\vol B)\sqrt{\vol B}} = x \, .
\end{multline*}

Item (b). We take $ p = \frac{\la}{\mu\sqrt2}$, note that $ p>1 $
(since $ \al \le -y < 0 $), $ \mu = \frac{\la}{p\sqrt2} $, and
$ \frac{p-1}{|\la|} = \frac1{|\mu|\sqrt2} - \frac1{|\la|}
= - \frac{\al}{\sqrt2} \ge \frac{y}{\sqrt2} = \frac{C_1}{\sqrt{\vol
B}} \log^{d-1} S(\vol B) $, that is, $ C_1
|\la| \le (p-1) \sqrt{\vol B} \log^{-(d-1)} S(\vol B) $. Using
Lemma \ref{2.1}(b),
\begin{multline*}
\be = \frac{ f_B(\la) }{ |\la| \sqrt{\vol B} } - \frac{
 f_{B/2}\(\frac{\la}{p\sqrt2}\)
 }{ \frac{|\la|}{p\sqrt2} \sqrt{\frac{\vol B}{2} } } =
 \frac1{ |\la| \sqrt{\vol B} } \bigg( f_B(\la) - 2p
 f_{B/2} \Big(\frac{\la}{p\sqrt2}\Big) \bigg) \ge \\
\ge \frac1{ |\la| \sqrt{\vol B} } \cdot
 (-C_1) \frac1{p-1} \frac{\la^2}{R(\vol B)} = - C_1
 |\la| \cdot \frac1{p-1} \cdot \frac1{R(\vol B)\sqrt{\vol B}} \ge \\
\ge - (p-1) \sqrt{\vol B} \log^{-(d-1)} S(\vol
 B) \cdot \frac1{p-1} \cdot \frac1{R(\vol B)\sqrt{\vol B}} = - x \, .
\quad \rlap{$\qedsymbol$}
\end{multline*}
\end{proof}

%% file: sect3.tex
\emph{Second digression on boxes.}
We iterate the box-halving operation $ B \mapsto B/2 $, getting $
B \mapsto B/2^n$:
\[
B/2^0 = B \, , \quad \text{and} \quad B/2^n = (B/2^{n-1})/2 \;\> \text{ for
} n=1,2,\dots \, ;
\]
clearly, $ \vol(B/2^n) = 2^{-n} \vol B $.

\begin{lemma}\label{3.1}
If a box $ B \subset \R^d $ and a number $ C>0 $ satisfy $
C \le \width B $, then there exists one and only one $
n \in \{0,1,2,\dots\} $ such that
\[
C \le \width(B/2^n) \le \length(B/2^n) < 2C \, ,
\]
and this $n$ satisfies $ 2^{-d} \vol B < C^d 2^n \le \vol B $.
\end{lemma}

\begin{sloppypar}
\begin{proof}
If $n$ fits (that is, $ C \le \width(B/2^n) \le \length(B/2^n) < 2C
$), then, using \eqref{2.05},
$ \width(B/2^{n+1}) \le \frac12 \length(B/2^n) < C $, which shows,
first, that $n$ (if exists) is unique (since $n+1$ and larger numbers
cannot fit), and second, that the least candidate is the least $n$
such that $ \width(B/2^{n+1}) < C $ (it exists, since
$ \width(B/2^n) \le (2^{-n}\vol B)^{1/d} \to 0 $ as $ n\to\infty
$). For this $n$ we have $ \width(B/2^{n+1}) = \frac12 \length(B/2^n)
$ (since $ \width(B/2^n) \ge C > \width(B/2^{n+1})
= \min \( \width(B/2^n), \frac12 \length(B/2^n) \) $ by \eqref{2.05}),
whence $ \length(B/2^n) = 2 \width(B/2^{n+1}) < 2C $, which shows that
this candidate fits. Finally, $ C \le \width(B/2^n) \le
(2^{-n}\vol B)^{1/d} \le \length (B/2^n) < 2C $, thus $ C^d \le
2^{-n} \vol B < (2C)^d $.
\end{proof}
\end{sloppypar}

\begin{corollary}\label{3.2}
If $ \width B \ge C $, then $ \width(B/2^n) \ge C $ for all $n$ such
that $ C^d 2^{n-1} \le 2^{-d} \vol B $ (since such $n$ cannot exceed
the $n$ of \ref{3.1}).
\end{corollary}

\emph{Back to random fields.}
We have a splittable CMS random field $X$ on $\R^d$ and the corresponding
functions $ f_B(\la) $ defined in Sect.~\ref{sect2}. We still use $
R(v) $, $ S(v) $ and $ C_1 $ introduced there, but now we assume that
$ C_1 \ge 3 $.

Our goal may be formulated using ``for large enough'' quantifiers as
follows.

\begin{proposition}\label{3.3}
For all $ \eps>0 $, for all $V$ large enough, for all $ \de>0 $, for
all $n$ large enough, for all boxes $ B \subset \R^d $ and numbers $a>0$, if
\begin{equation}\label{3.4}
\begin{gathered}
a \le \frac1\eps \, , \quad 2^n V \le \vol B \le \frac1\eps 2^n V \,
 , \quad \width B \ge C_1 \quad \text{and} \\
f_{B/2^n} (\la) \le a \la^2 \quad \text{for all } \la \in
 [-\de,\de] \, ,
\end{gathered}
\end{equation}
then
\begin{equation}\label{3.5}
\begin{gathered}
f_B (\la) \le (a+\eps) \la^2  \quad \text{for all } \la \in
 [-\De,\De] \, , \quad \text{where} \\
\De = \frac1{C_1 \sqrt{a+\eps}} \sqrt{S(\vol B)} \log^{-(d-1)} S(\vol B) \, .
\end{gathered}
\end{equation}
\end{proposition}

Symbolically:
\[
\forall \eps>0 \;\, \forall^\infty V \;\, \forall \de>0 \;\, \forall^\infty
n \;\, \forall B \;\, \forall
a>0 \;\> \( \eqref{3.4} \impl \eqref{3.5} \) \, .
\]

In other words, the implication ``\eqref{3.4} \imp \eqref{3.5}''
holds \emph{ultimately} in the following sense: it admits a sufficient
condition of the form\footnote{%
 This is the meaning of the word ``ultimately'' throughout this
 section (and only in this section). Existence of functions $
 V_0(\cdot) $, $ N(\cdot) $ will be proved; their regularity
 properties (like continuity, monotonicity etc.) are irrelevant (but
 hold).}
\[
V > V_0(\eps,d,C_1) \, , \quad
n \ge N(\de,V,\eps,d,C_1) \, .
\]

To dispel any doubts, here is a more traditional formulation.

\begin{customprop}{\ref*{3.3}a}\label{eight}
For every $ \eps > 0 $ there exists $ V_\eps $ such that for all $
V \in [V_\eps,\infty) $ and $ \de>0 $ there exists natural $N$ such
that \eqref{3.5} holds for all $ n \in \{N,N+1,\dots\} $, $ a \in
(0,\infty) $ and boxes $B$ satisfying \eqref{3.4}.
\end{customprop}

Note that $X$ is given, thus $ V_\eps $ and $ N $ may depend on $ d $
and $ C_1 $.

\begin{lemma}\label{3.6}
It is sufficient to prove Prop.~\ref*{3.3}a assuming in addition $
a \ge \eps $.
\end{lemma}

\begin{proof}
WLOG, $ \eps\le1 $. We consider formulas (\ref*{3.4}$\,'$),
(\ref*{3.5}$\,'$) obtained from \eqref{3.4}, \eqref{3.5} by replacing
$ \eps,a $ with $ \ti\eps = \frac\eps2 $, $ \ti a = a + \frac\eps2
$. We have \eqref{3.4} \imp (\ref*{3.4}$\,'$), since $ \ti\eps \le \eps
$, $ \ti a \ge a $ and $ \ti\eps \ti a = \frac{\eps a}2
+ \frac{\eps^2}4 \le \frac12 + \frac14 \le 1 $. Also,
(\ref*{3.5}$\,'$) \imp \eqref{3.5}, since $ \ti a + \ti\eps = a+\eps
$. By assumption, (\ref*{3.4}$\,'$) \imp (\ref*{3.5}$\,'$) (since $ \ti
a \ge \ti\eps $), and we get \eqref{3.4} \imp \eqref{3.5}.
\end{proof}

\[
\includegraphics[scale=1]{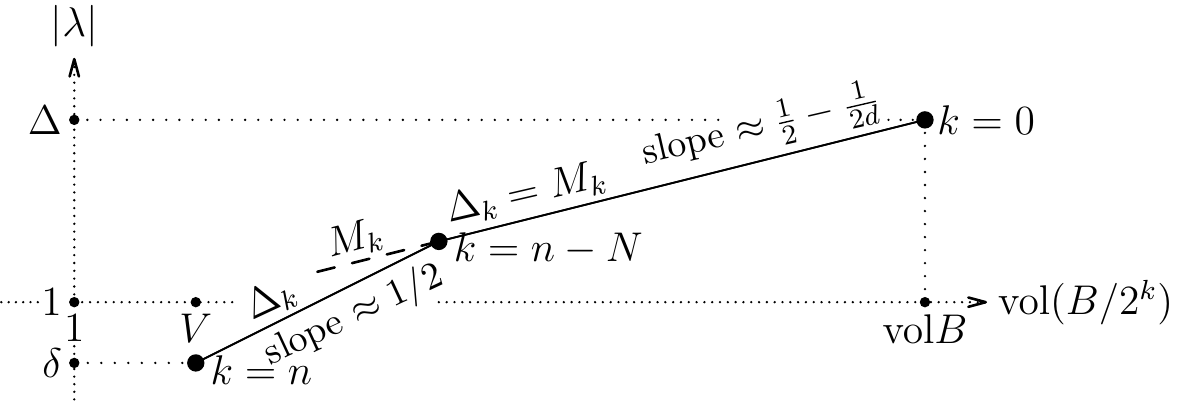}
\]
\begin{center}A hint to the proof of Prop.~\ref*{3.3}a (log--log plot).\end{center}

\begin{proof}[Proof of Prop.~\ref*{3.3}a]
According to Lemma \ref{3.6}, we assume $ a \ge \eps $.

Given $ \eps $, $ V $, $ \de $, $ n $, $ B $ and $ a $
satisfying \eqref{3.4}, we note that ultimately
$ \width(B/2^{n-1}) \ge C_1 $ by Corollary \ref{3.2} (since $
2^{-d} \vol B \ge 2^{-d} \cdot 2^n V > 2^{n-2}
C_1^d $ ultimately), thus Lemma \ref{2.2}(a) may be
applied to $ B/2^k $ for each $ k=0,\dots,n-1 $ provided that numbers
$u_k$ and $\de_k$ are given. In order to prove by induction in $ k =
n-1,n-2,\dots,0 $ the inequality
\begin{equation}\label{3.6a}
f_{B/2^k} (\la) \le a_k \la^2  \quad \text{for all } \la \in
 [-\De_k,\De_k] 
\end{equation}
we need $ a_0,\dots,a_n \ge 0 $ and $ \De_0,\dots,\De_n \ge 0 $ such
that for all $ k=0,\dots,n-1 $
\begin{gather}
\sqrt{a_k} \ge \sqrt{a_{k+1}} + x_k \, , \label{3.7} \\
\De_k \le  \min \bigg( \frac{\sqrt2}{p_k} \De_{k+1}, \frac1{C_1}
 \frac{p_k-1}{p_k} \sqrt{\vol(B/2^k)} \log^{-(d-1)}
 S(\vol(B/2^k)) \bigg) \label{3.8}
\end{gather}
where $ p_k = \frac{\sqrt{a_{k+1}}+x_k}{\sqrt{a_{k+1}}} $ and $ x_k
= \sqrt{\frac{C_1}{R(\vol(B/2^k))}} $; we also need $ a_n \ge a $,
$ \De_n \le \de $ in order to get \eqref{3.6a} for $ k=n $
from \eqref{3.4}; and $ a_0 \le a+\eps $,
$ \De_0 \ge \De $ in order to get \eqref{3.5} from \eqref{3.6a} for
$k=0$.

We take $ \sqrt{a_k} = \sqrt a + \sqrt{ \frac{C_1}{R(\vol(B/2^n))}
} \sum_{i=1}^{n-k} 2^{-\frac{i}{2d}} $ (thus $ \sqrt{a_k}
= \sqrt{a_{k+1}}+x_k $, since $ x_k = 2^{-(n-k)/(2d)} x_n $) and note
that $ \sqrt{a_0} \le \sqrt a + \sqrt{ \frac{C_1}{V^{1/d}} }
\sum_{i=1}^\infty 2^{-\frac{i}{2d}} $ (since $ R(\vol(B/2^n)) \ge
R(V) = V^{1/d} $). Ultimately $ \sqrt{ \frac{C_1}{V^{1/d}} } \sum_{i=1}^\infty
2^{-\frac{i}{2d}} \le \sqrt{\eps^{-1}+\eps}-\sqrt{\eps^{-1}} $, thus
$ a_0 \le a+\eps $ (since $ a \le \eps^{-1} $ and therefore
$ \sqrt{\eps^{-1}+\eps}-\sqrt{\eps^{-1}} \le \sqrt{a+\eps} - \sqrt a
$).

Now about $ \De_k $. We introduce $ M_k = \frac1{C_1} \sqrt{ \frac1a
S(\vol(B/2^k)) } \log^{-(d-1)} S(\vol(B/2^k)) $ and want to replace
the condition \eqref{3.8} with a stronger condition
\begin{equation}\label{3.10}
\De_k \le \min \bigg( \frac{\sqrt2}{p_k} \De_{k+1}, M_k \bigg) \, ;
\end{equation}
to this end we prove that \eqref{3.10} is stronger, that is,\\
$ M_k \le \frac1{C_1} \frac{p_k-1}{p_k} \sqrt{\vol(B/2^k)} \log^{-(d-1)}
S(\vol(B/2^k)) $. We rewrite this inequality as $ \sqrt{ \frac1a
S(\vol(B/2^k)) } \le \frac{p_k-1}{p_k} \sqrt{\vol(B/2^k)} $, that is,
\[
\frac{p_k}{p_k-1} \le \sqrt{ a R(\vol(B/2^k)) } \, .
\]
We note that $ \frac{p_k}{p_k-1} = 1 + \frac{\sqrt{a_{k+1}}}{x_k} = 1
+ \sqrt{ \frac{a_{k+1} R(\vol(B/2^k))}{C_1} } \le 1
+ \sqrt{ \frac{(a+\eps) R(\vol(B/2^k))}{C_1} } $ and prove that the
latter does not exceed $ \sqrt{aR(\vol(B/2^k)) } $, that is, $
1 \le \sqrt{aR(\vol(B/2^k))} ( 1 - \sqrt{ \frac{a+\eps}{C_1 a} } )
$. We know that $ a \ge \eps $ and $ C_1 \ge 3 $, thus
$ \frac{a+\eps}{C_1 a} \le \frac2{C_1} \le \frac23 $, and $
aR(\vol(B/2^k)) \ge \eps R(V) $. Ultimately $ \sqrt{ \eps
V^{1/d} } ( 1 - \sqrt{\frac23} ) \ge 1 $, thus \eqref{3.10} is
stronger than \eqref{3.8}.

We note that $ p_k-1 = \frac{x_k}{\sqrt{a_{k+1}}} = \sqrt{
\frac{ C_1 }{ a_{k+1} R(\vol(B/2^k)) } } \le \sqrt{ \frac{C_1}{a}
} \frac1{ \sqrt{ R(2^{n-k}\vol(B/2^n)) } } \le \sqrt{ \frac{C_1}{\eps}
} (2^{n-k}V)^{-\frac1{2d}} \le \sqrt{ \frac{C_1}{\eps V^{1/d}}
} \cdot 2^{-\frac{n-k}{2d}} $.
Ultimately $ \eps V^{1/d} \ge C_1 $, thus $ p_k - 1 \le
2^{-\frac{n-k}{2d}} $. Ultimately $ 2^{-\frac{N+1}{2d}} \le
2^{\frac1{2d}}-1 $ (where $ N = N(\de,V,\eps,d,C_1) $), thus $
M_k \le \frac{ \sqrt2 }{ p_k } M_{k+1} $ for $ k = 0,\dots,n-N-1 $
(since $ \frac{M_k}{M_{k+1}} \le 2^{\frac{d-1}{2d}} $ and $ p_k \le 1
+ 2^{-\frac{n-k}{2d}} \le 1 + 2^{-\frac{N+1}{2d}} \le 2^{\frac1{2d}}
$.)

We choose $ \De_k $ as follows:
\[
\begin{cases}
\De_k = M_k &\text{for } k=0,\dots,n-N;\\
\De_k
= \frac{p_{k-1}}{\sqrt2} \frac{p_{k-2}}{\sqrt2} \dots \frac{p_{n-N}}{\sqrt2}
M_{n-N} &\text{for } k=n-N+1,\dots,n.
\end{cases}
\]
Clearly, $ \De_0 = M_0 \ge \De $. It remains to prove that
$ \De_n \le \de $, and \eqref{3.10}.

\begin{sloppypar}
We note that $ p_{n-N} \dots p_{n-1} \le \prod_{k=1}^N ( 1 +
2^{-\frac{k}{2d}} ) \le \exp \sum_{k=1}^\infty 2^{-\frac{k}{2d}} $ and
(using \eqref{3.4}) $ C_1 M_{n-N} \le \sqrt{ \frac1a
S(\vol(B/2^{n-N})) } \le \sqrt{ \frac1a S(\frac1\eps \frac{2^n
V}{2^{n-N}} ) } $, thus $ \De_n
= \frac{p_{n-1}}{\sqrt2} \dots \frac{p_{n-N}}{\sqrt2} M_{n-N} \le
2^{-N/2} ( \exp \sum_{k=1}^\infty 2^{-\frac{k}{2d}}
) \frac1{C_1} \sqrt{ \frac1a S(\frac{2^N V}\eps) } = 2^{-\frac N{2d} }
( \exp \sum_{k=1}^\infty 2^{-\frac{k}{2d}} ) \frac1{C_1 \sqrt
a} \sqrt{S(V/\eps)} $. Ultimately $ 2^{-\frac N{2d} }
( \exp \sum_{k=1}^\infty 2^{-\frac{k}{2d}}
) \frac1{C_1 \sqrt\eps} \sqrt{S(V/\eps)} \le \de $, thus
$ \De_n \le \de $.

We have $ \De_k \le \frac{\sqrt2}{p_k} \De_{k+1} $ for all $k$ (since
$ \De_k = M_k \le \frac{\sqrt2}{p_k} M_{k+1}
= \frac{\sqrt2}{p_k} \De_{k+1} $ for $ k=0,\dots,n-N-1 $, and
$ \De_{k+1} = \frac{p_k}{\sqrt2} \De_k $ for $k=n-N,\dots,n-1 $).
Thus, in order to prove \eqref{3.10} it is sufficient to check that
$ \De_k \le M_k $ for $ k = n-N+1,\dots,n $. We'll get a bit more:
$ \De_k \le \frac1{C_1} \sqrt{ \frac1a S(\vol(B/2^k)) } \log^{-(d-1)}
S(\vol(B/2^{n-N})) $, that is,
\begin{gather*}
\frac{p_{k-1}}{\sqrt2} \frac{p_{k-2}}{\sqrt2} \dots \frac{p_{n-N}}{\sqrt2} \sqrt{
S(\vol(B/2^{n-N})) } \le \sqrt{ S(\vol(B/2^k)) } \, ; \\
p_{n-N} \dots p_{k-1} \le 2^{ \frac{k-n+N}2 } \cdot 2^{-\frac{d-1}{2d}
(k-n+N)} = 2^{ \frac{k-n+N}{2d} } \, .
\end{gather*}
We note that $ p_k \le p_{k+1} $ for all $k$ (since $ p_k = 1 + \sqrt{
\frac{ C_1 }{ a_{k+1} R(\vol(B/2^k)) } } $ and $ a_k \ge a_{k+1} $),
therefore the product $ p_{n-N} \dots p_{k-1} $ is a logarithmically
convex function of $k$, and we may check the inequality $
p_{n-N} \dots p_{k-1} \le 2^{ \frac{k-n+N}{2d} } $ only for $ k=n-N $
and $ k=n $. For $ k=n-N $ it is just $ 1 \le 1 $. For $ k=n $ we need
$ p_{n-N} \dots p_{n-1} \le 2^{ \frac{N}{2d} } $; it remains to note
that ultimately $ 2^{ \frac{N}{2d} } \ge \exp \sum_{k=1}^\infty
2^{-\frac{k}{2d}} $.
\end{sloppypar}
\end{proof}

Here is the (quite similar) lower bound.

\begin{proposition}\label{3.11}
For all $ \eps>0 $, for all $V$ large enough, for all $ \de>0 $, for
all $n$ large enough, for all boxes $ B \subset \R^d $ and numbers
$a>0$, if
\begin{equation}\label{3.4'}\tag{\ref*{3.4}$\,'$}
\begin{gathered}
a \le \frac1\eps \, , \quad 2^n V \le \vol B \le \frac1\eps 2^n V \,
 , \quad \width B \ge C_1 \quad \text{and} \\
f_{B/2^n} (\la) \ge a \la^2 \quad \text{for all } \la \in
 [-\de,\de] \, ,
\end{gathered}
\end{equation}
then
\begin{equation}\label{3.5'}\tag{\ref*{3.5}$\,'$}
\begin{gathered}
f_B (\la) \ge (a-\eps) \la^2  \quad \text{for all } \la \in
 [-\De,\De] \, , \quad \text{where} \\
\De = \frac1{C_1 \sqrt{a}} \sqrt{S(\vol B)} \log^{-(d-1)} S(\vol B) \, .
\end{gathered}
\end{equation}
\end{proposition}

\begin{proof}
We assume $ a \ge \eps $ (otherwise the conclusion is void). We apply
Lemma \ref{2.2}(b) similarly to the proof of Prop. \ref*{3.3}a; the
logic is the same, but formulas are a bit different:
\begin{gather*}
f_{B/2^k} (\la) \ge a_k \la^2  \quad \text{for all } \la \in
 [-\De_k,\De_k] 
\, , \tag{\ref*{3.6a}$\,'$}\label{3.6a'} \\
\sqrt{a_k} \le \sqrt{a_{k+1}} - x_k \, , \tag{\ref*{3.7}$\,'$}\label{3.7'} \\
\!\!\!\!
\De_k \le  \min \!\bigg(\! p_k \sqrt2 \De_{k+1}, \frac1{C_1}
 (p_k-1) \sqrt{\vol(B/2^k)} \log^{-(d-1)}
 S(\vol(B/2^k)) \!\bigg)\! \!\!\!\!
 \tag{\ref*{3.8}$\,'$}\label{3.8'}
\end{gather*}
where $ p_k = \frac{\sqrt{a_{k+1}}}{\sqrt{a_{k+1}}-x_k} $ (and $x_k$
as before); we also need $ a_n \le a $, $ \De_n \le \de $; and $
a_0 \ge a-\eps $, $ \De_0 \ge \De $.

We take $ \sqrt{a_k} = \sqrt a - \sqrt{ \frac{C_1}{R(\vol(B/2^n))}
} \sum_{i=1}^{n-k} 2^{-\frac{i}{2d}} $ and note that
$ \sqrt{a_0} \ge \sqrt a - \sqrt{ \frac{C_1}{V^{1/d}}
} \sum_{i=1}^\infty 2^{-\frac{i}{2d}} $. Ultimately
$ \sqrt{ \frac{C_1}{V^{1/d}} } \sum_{i=1}^\infty
2^{-\frac{i}{2d}} \le \sqrt{\eps^{-1}}-\sqrt{\eps^{-1}-\eps} $, thus
$ a_0 \ge a-\eps $.

Now about $ \De_k $. We borrow $ M_k $ and $ \De_k $ from the proof of
Prop. \ref*{3.3}a (``as is'', without replacing $p_k$ used there with
$p_k$ used here). We want to replace the condition \eqref{3.8'} with a
stronger condition
\begin{equation}\tag{\ref*{3.10}$\,'$}\label{3.10'}
\De_k \le \min \bigg( p_k \sqrt2 \De_{k+1}, M_k \bigg) \, ;
\end{equation}
to this end we prove that $ M_k \le \frac1{C_1}
(p_k-1) \sqrt{\vol(B/2^k)} \log^{-(d-1)} S(\vol(B/2^k)) $, that is,
\[
\frac1{p_k-1} \le \sqrt{ a R(\vol(B/2^k)) } \, .
\]
We have $ \frac1{p_k-1} = -1 + \sqrt{ \frac{a_{k+1}
R(\vol(B/2^k))}{C_1} } \le \sqrt{ a R(\vol(B/2^k)) } $, since $
a_{k+1} \le a $ (and $ C_1 \ge 1 $). Thus, \eqref{3.10'} is stronger
than \eqref{3.8'} (this time, no need to bother about $ V $).

As was seen in the proof of Prop. \ref*{3.3}a, $ \De_n \le \de $,
$ \De_0 \ge \De $ (now $ \De = \frac1{C_1 \sqrt a} \dots $ rather than
$ \frac1{C_1 \sqrt{a+\eps}} \dots $ since Lemma \ref{3.6} is now
irrelevant), and $ \De_k \le M_k $; also
$ \De_k \le \frac{ \sqrt2 }{ \ti p_k } \De_{k+1} $ where $ \ti p_k $
is the $ p_k $ used there. Anyway, $ \ti p_k \ge 1 $ and $ p_k \ge 1
$, thus $ \De_k \le \frac{ \sqrt2 }{ \ti p_k
} \De_{k+1} \le \sqrt2 \De_{k+1} \le p_k \sqrt2 \De_{k+1} $, which
gives \eqref{3.10'}.
\end{proof}

%% file: sect4.tex
The constant $ C_1 $, the functions $ R(\cdot) $, $ S(\cdot) $ and the
random field $X$ are the same as in Sect.~\ref{sect3}.

\begin{proposition}\label{4.1}
For all $ \eps>0 $ and $ W \ge C_1 $ there exists $ C>0 $ such that
for all boxes $ B \subset \R^d $ and numbers $\la\in\R$ such that
\begin{gather}
\vol B \ge C \, , \quad \width B \ge W \, , \quad
 \text{and} \label{4.2a} \\
0 < C |\la| \le \sqrt{\vol B} \log^{-d} \vol B \label{4.3a}
\end{gather}
there exist $n\in\{0,1,2,\dots\}$ and $\mu\in\R$ such that
\begin{gather}
\width(B/2^n) \ge W \, , \quad
\frac1{\la^2} f_B(\la) \le (1+\eps) \frac1{\mu^2} f_{B/2^n}(\mu) +
 \eps \, , \quad \text{and} \label{4.4} \\
|\mu| \le \eps \sqrt{S(\vol(B/2^n))} \log^{-(d-1)} \vol(B/2^n) \,
 . \label{4.5}
\end{gather}
\end{proposition}

We start proving Prop.~\ref{4.1}.

\[
\includegraphics[scale=1]{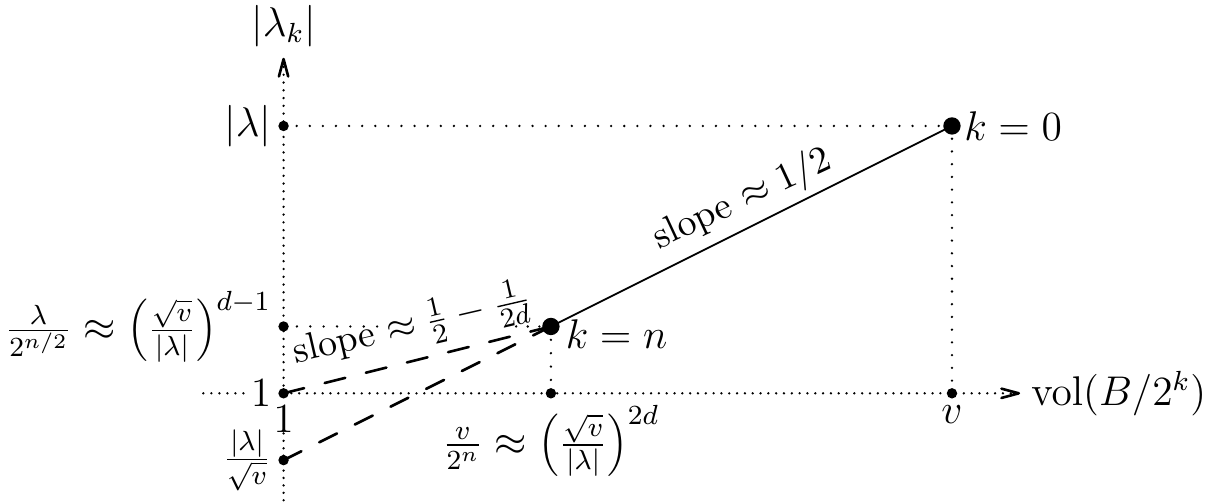}
\]
\begin{center}A hint to the proof (log--log plot).\end{center}

We'll find $C_2$ and $C_3$ such that \eqref{4.4}, \eqref{4.5} follow
from
\begin{gather}
\vol B \ge C_2 \, , \quad \width B \ge W \, , \quad
 \text{and} \tag{\ref*{4.2a}$\,'$} \label{4.2} \\
0 < C_3 |\la| \le \sqrt{\vol B} \log^{-d} \vol B \, ;
 \tag{\ref*{4.3a}$\,'$} \label{4.3}
\end{gather}
then we just take $ C = \max(C_2,C_3) $.

We'll prove that the claim ``for all $B,\la$ satisfying \eqref{4.2},
\eqref{4.3} there exist $n,\mu$ satisfying \eqref{4.4}, \eqref{4.5}''
holds \emph{ultimately} in the following sense: it admits a sufficient
condition of the form\footnote{%
 This is the meaning of the word ``ultimately'' throughout this
 section (and only in this section).}
\[
C_3 \ge C_{3,0}(\eps,W,d,C_1) \, , \quad
C_2 \ge C_{2,0}(C_3,\eps,W,d,C_1) \, .
\]

We assume that
\begin{equation}\label{4.6}
|\la| > \eps \sqrt{S(\vol B)} \log^{-(d-1)} \vol B \, ,
\end{equation}
since otherwise $n=0$, $\mu=\la$ do the job.
WLOG, $ \eps C_1 \le 1 $.

Given $n$ such that $ \width(B/2^{n-1}) \ge C_1 $ and numbers $ \la_0,
\dots, \la_n $, all positive or all negative, such that
\begin{equation}\label{4.7}
\frac{\sqrt2}{|\la_k|} - \frac1{|\la_{k+1}|} \ge y_k \quad \text{for all
} k=0,\dots,n-1 \, ,
\end{equation}
where
\[
y_k = \frac{C_1}{\sqrt{ 2^{-(k+1)} v }} \log^{d-1} S(2^{-k} v)
\]
and $ v $ denotes $ \vol B $, Lemma \ref{2.3}(a) gives
\begin{equation}\label{4.8}
\frac1{ |\la_k| \sqrt{ 2^{-k} v } } f_{B/2^k} (\la_k) \le \frac1{
|\la_{k+1}| \sqrt{ 2^{-(k+1)} v }} f_{B/2^{k+1}} (\la_{k+1}) + x_k \, ,
\end{equation}
where
\[
x_k = \frac1{ R(2^{-k} v) } \log^{-(d-1)} S(2^{-k} v) \, .
\]
In order to obtain \eqref{4.4}, \eqref{4.5} we need $n$ such that
$ \width(B/2^n) \ge W, $ and $ \la_0,
\dots, \la_n $ (as above) such that $ \la_0 = \la $ and, taking $ \mu
= \la_n $, we get \eqref{4.5}, $ 2^{n/2} |\mu| \le (1+\eps) |\la| $,
and
\begin{equation}\label{4.9}
x_0 + \dots + x_{n-1} \le \eps \frac{ |\la| }{ \sqrt v } \, ;
\end{equation}
then \eqref{4.8} gives $ \frac{ f_B(\la) }{ |\la| \sqrt v } \le \frac{
f_{B/2^n}(\mu) }{ |\mu| \sqrt{2^{-n}v} } + \eps \frac{ |\la| }{ \sqrt
v } $, which implies the second inequality of \eqref{4.4}:
\[
\frac1{\la^2} f_B(\la) \le \frac{ \sqrt v }{ |\la| } \Big( \frac{
f_{B/2^n}(\mu) }{ |\mu| \sqrt{2^{-n}v} } + \eps \frac{ |\la| }{
\sqrt v } \Big) = \frac{ 2^{n/2} |\mu| }{ |\la| } \cdot \frac{
f_{B/2^n}(\mu) }{ \mu^2 } + \eps \, .
\]

We note that $ \frac{\sqrt v}{ C_3 |\la| } \ge \log^d v $
by \eqref{4.3} and $ v \ge C_2 $ by \eqref{4.2}, whence
\begin{equation}\label{***}
\frac{\sqrt v}{ C_3 |\la| } \to \infty \quad \text{ultimately}
\end{equation}
in the following sense: for every $ M \in \R $ the inequality
$ \frac{\sqrt v}{ C_3 |\la| } \ge M $ ultimately holds. This means
existence of functions $ C_{3,1} $ and $ C_{2,1} $ such that
$ \frac{\sqrt v}{ C_3 |\la| } \ge M $ whenever $ C_3 \ge C_{3,1}
(M,\eps,W,d,C_1) $ and $ C_2 \ge C_{2,1} (M,C_3,\eps,W,d,C_1)
$. Moreover, the number $M$ may be replaced with an arbitrary function
of $ \eps,W,d,C_1 $; indeed, $ \frac{\sqrt v}{ C_3 |\la| } \ge
M(\eps,W,d,C_1) $ whenever $ C_3 \ge C_{3,1}
(M(\eps,W,d,C_1),\eps,W,d,C_1) $ and $ C_2 \ge C_{2,1}
(M(\eps,W,d,C_1),C_3,\eps,W,d,C_1) $. For example, ultimately $
C_3 \to \infty $, $ \frac{C_2}{C_3} \to \infty $, but $C_1$ and $W$ do
not tend to infinity.

We take integer $n$ such that
\begin{equation}\label{4.10}
2^{n-1} < (2d)^{2d(d-1)} (C_3 |\la|)^{2d} v^{-(d-1)}
\log^{2d(d-1)} \frac{\sqrt v}{C_3 |\la|} \le 2^n \, .
\end{equation}

\begin{lemma}\label{4.11}
$ n \ge 1 $ ultimately.
\end{lemma}

\begin{sloppypar}
\begin{proof}
Assume the contrary: $ (2d)^{2d(d-1)} (C_3 |\la|)^{2d}
v^{-(d-1)} \log^{2d(d-1)} \frac{\sqrt v}{C_3|\la|} \le 1 $, that is,
\begin{equation}\label{*3.10}
(2d)^{d-1} \frac{ C_3 |\la| }{ \sqrt{S(v)} } \log^{d-1} \frac{\sqrt
v}{C_3 |\la|} \le 1 \, .
\end{equation}
Ultimately, first, $ |\la|>\eps $ due to \eqref{4.6}, second, $ \eps
C_3 > 1 $, and therefore, third, $ C_3 |\la| > 1 $, which
contradicts \eqref{*3.10} for $d=1$. Now consider $ d \ge 2
$. Using \eqref{4.6} and \eqref{*3.10},
\begin{gather*}
\eps \frac{ \sqrt{S(v)} }{ \log^{d-1} v } < |\la| \le \frac1{C_3 (2d)^{d-1}}
 \sqrt{S(v)} \log^{-(d-1)} \frac{\sqrt v}{C_3 |\la|} \, ; \\
\log^{d-1} v > \eps C_3 (2d)^{d-1} \log^{d-1} \frac{\sqrt v}{C_3 |\la|} =
\eps C_3 \Big( 2d \log \frac{\sqrt v}{C_3 |\la|} \Big)^{d-1} \, ;
\end{gather*}
ultimately $ \eps C_3 \ge 1 $, and we get
$ \log^{d-1} v > (2d)^{d-1} \log^{d-1} \frac{\sqrt v}{C_3|\la|} $;
$ v > \frac{v^d}{(C_3|\la|)^{2d}} $;
$ 1 > \frac{v^{d-1}}{(C_3|\la|)^{2d}}
= \( \frac{\sqrt{S(v)}}{C_3|\la|} \)^{2d} $.
Now \eqref{*3.10} gives $ (2d)^{d-1} \log^{d-1}
\frac{\sqrt v}{C_3 |\la|} \le 1 $, which ultimately
contradicts \eqref{***}.
\end{proof}
\end{sloppypar}

Below, ``$y\le\const(\al,\be)x$'' means that $ y \le \const \cdot x $
for some constant dependent on $\al,\be$ only.

\begin{lemma}\label{4.13}
$\displaystyle  n \le \frac{\eps}{1+\eps} \cdot \frac{\sqrt v}{C_1
|\la|} \log^{-(d-1)} v $ ultimately.
\end{lemma}

\begin{proof}
First we prove that $ n \le \const(d) \log v $.
  
We have $ \frac{\eps}{|\la|} \le \const(d) $, since by \eqref{4.6}, $
\frac{\eps}{|\la|} < \frac{ \log^{d-1} v }{ \sqrt{S(v)} } $, the latter
being bounded in $ v \in (1,\infty) $.

Ultimately $ \eps C_3 \ge 1 $, thus
$ \frac1{C_3 |\la|} = \frac1{\eps C_3} \cdot  \frac{\eps}{|\la|} \le
\frac{\eps}{|\la|} \le \const(d) $,
whence, using \eqref{***} and \eqref{4.3},
$ 1 \le \frac{\sqrt v}{C_3 |\la|} \le \const(d) \sqrt v $.
Using this and \eqref{4.10} (the left-hand inequality) we get
\begin{multline*}
2^n < 2 (2d)^{2d(d-1)} (C_3 |\la|)^{2d} \frac1{v^{d-1}} \log^{2d(d-1)}
 \frac{\sqrt v}{C_3 |\la|} \le \\
\le \const(d) (\sqrt v)^{2d} \frac1{v^{d-1}}
 \log^{2d(d-1)} ( \const(d) \sqrt v ) \le \\
\le \const(d) v \Big( \const(d) +
 \frac12 \log v \Big)^{2d(d-1)} \le \const(d) v^2 \, ,
\end{multline*}
whence $ n \le \const() \log \( \const(d) v^2 \) \le \const(d) \log
v $.

Second, using $ n \le \const(d) \log v $ and then \eqref{4.3} we have
\begin{multline*}
n \frac{1+\eps}{\eps} \cdot \frac{C_1 |\la|}{\sqrt v} \log^{d-1} v \le
 \frac{1+\eps}{\eps} \cdot \frac{C_1 |\la|}{\sqrt v} \const(d) \log^d v
= \\
\frac{1+\eps}{\eps C_3} \cdot \frac{C_1}{\sqrt v} \const(d) C_3
|\la| \log^d v \le
\frac{1+\eps}{\eps C_3} \cdot \frac{C_1}{\sqrt v} \const(d) \sqrt v =
\frac{1+\eps}{\eps C_3} \cdot C_1 \const(d) \, ;
\end{multline*}
and $ \eps C_3 \ge \const(d) (1+\eps) C_1 $ ultimately.
\end{proof}

Having $ n \log^{d-1} v \le \frac{\eps}{1+\eps} \cdot \frac{\sqrt
v}{C_1 |\la|} $ (ensured by Lemma \ref{4.13}) we define $
\la_0,\dots,\la_n $ (all positive or all negative) by $ \la_0 = \la $
and
\[
\frac{\sqrt2}{|\la_k|} - \frac1{|\la_{k+1}|} = y_k \quad \text{for all
} k=0,\dots,n-1 \, .
\]
That is,
\[
\frac1{ 2^{k/2} |\la_k| } = \frac1{|\la|} - \frac{C_1}{\sqrt v}
\sum_{i=0}^{k-1} \log^{d-1} S(2^{-i}v) \, ;
\]
the right-hand side is positive, since
\[
\sum_{i=0}^{k-1} \log^{d-1} S(2^{-i}v) \le k \log^{d-1} S(v) \le n
\log^{d-1} v \le \frac{\eps}{1+\eps} \cdot \frac{\sqrt v}{C_1
|\la|} \, ;
\]
moreover, for $ k = n $ we get an inequality required
before \eqref{4.9}:
\begin{equation}\label{4.14}
2^{n/2} |\mu| \le (1+\eps) |\la| \, .
\end{equation}

\begin{lemma}\label{4.15}
$\displaystyle 2^{-n} v \to \infty $ ultimately.
\end{lemma}

\begin{proof}
By \eqref{4.10} (the left-hand inequality),
\[
\!\!\!\! 2^{-n} v > \frac{ v \cdot v^{d-1} }{ 2(2d)^{2d(d-1)}
(C_3 |\la|)^{2d} } \log^{-2d(d-1)} \frac{\sqrt v}{C_3 |\la|} =
\const(d) \Big( \frac{\sqrt v}{C_3
|\la|} \log^{-(d-1)} \frac{\sqrt v}{C_3 |\la|} \Big)^{2d} ,
\]
which ultimately tends to infinity due to \eqref{***}, since $
x \log^{-(d-1)} x \to \infty $ as $ x \to +\infty $.
\end{proof}

Now we are in position to ensure that $ \width(B/2^{n-1}) \ge C_1
$ (as required before \eqref{4.7}) and $ \width(B/2^n) \ge W
$ (as required in \eqref{4.4}). The former follows from the latter. By
Corollary \ref{3.2} it is sufficient to have $ \width B \ge W $ (which
is given) and $ W^d 2^{n-1} \le 2^{-d} v $, that is, $ 2^{-n} v \ge
2^{d-1} W^d $ (which ultimately holds by Lemma \ref{4.15}, recall the
remark after \eqref{***}).

Here is the proof of \eqref{4.5}.

\begin{lemma}\label{4.16}
$\displaystyle |\mu| \le \eps \sqrt{S(2^{-n}v)} \log^{-(d-1)}
(2^{-n}v) $ ultimately.
\end{lemma}

\begin{sloppypar}
\begin{proof}
Assume the contrary. Using \eqref{4.14}, $ (1+\eps) 2^{-n/2} |\la| \ge
|\mu| > \eps \sqrt{S(2^{-n}v)} \log^{-(d-1)} (2^{-n}v) = \eps \cdot
2^{-n/2} \cdot 2^{\frac{n}{2d}} v^{\frac{d-1}{2d}} \log^{-(d-1)}
(2^{-n}v) $, whence $ 2^{\frac{n}{2d}} < \frac{1+\eps}{\eps} |\la|
v^{-\frac{d-1}{2d}} \log^{d-1} (2^{-n}v) $; using \eqref{4.10} (the
right-hand inequality),
\begin{multline*}
(2d)^{2d(d-1)} (C_3 |\la|)^{2d} v^{-(d-1)} \log^{2d(d-1)} \frac{\sqrt
 v}{C_3 |\la|} \le 2^n < \\
< \Big(\frac{1+\eps}{\eps}|\la|\Big)^{2d} v^{-(d-1)} \log^{2d(d-1)}
(2^{-n}v) \, ,
\end{multline*}
 \[
 (2d)^{d-1} \log^{d-1} \frac{\sqrt v}{C_3 |\la|} < \frac{1+\eps}{\eps
 C_3} \log^{d-1} (2^{-n}v) \, ;
\]
ultimately $ \frac{1+\eps}{\eps C_3} \le 1 $, and we get
\[
(2d)^{d-1} \log^{d-1} \frac{\sqrt v}{C_3 |\la|} < \log^{d-1} (2^{-n}v)
\, .
\]
For $d=1$ it means $1<1$. Now consider $ d \ge 2 $. We have $ 2d \log
\frac{\sqrt v}{C_3 |\la|} < \log (2^{-n}v) $; $ \( \frac{\sqrt v}{C_3
|\la|} \)^{2d} < 2^{-n} v $; $ 2^n < v \( \frac{C_3 |\la|}{\sqrt v}
\)^{2d} = \frac{ (C_3 |\la|)^{2d} }{ v^{d-1} } $. Using again
\eqref{4.10} (the right-hand inequality), $ (2d)^{2d(d-1)} \log^{2d(d-1)}
\frac{\sqrt v}{C_3 |\la|} < 1 $, which ultimately
contradicts \eqref{***}.
\end{proof}
\end{sloppypar}

We strengthen \eqref{***}:
\begin{equation}\label{***a}
\frac{\sqrt v}{ \phi(C_3,\eps,W,d,C_1) |\la| } \to \infty
\quad \text{ultimately}
\end{equation}
for every positive-valued function $\phi$. The proof is based again on
the inequality $ \frac{\sqrt v}{ C_3 |\la| } \ge \log^d C_2 $. For
arbitrary $M$ we have
\begin{multline*}
\frac{\sqrt v}{ \phi(C_3,\eps,W,d,C_1) |\la| } \ge M \> \Longleftarrow \>
\frac{C_3}{\phi(C_3,\eps,W,d,C_1)} \log^d C_2 \ge M \Longleftarrow \> \\
C_2 \ge \exp \bigg( \Big( \frac{ M \phi(C_3,\eps,W,d,C_1) }{ C_3
 } \Big)^{1/d} \bigg) \, ;
\end{multline*}
the latter holds ultimately.

Here is the proof of \eqref{4.9}.

\begin{lemma}\label{4.17}
$\displaystyle x_0 + \dots + x_{n-1} \le \eps \frac{ |\la| }{ \sqrt v } $.
\end{lemma}

\begin{sloppypar}
\begin{proof}
\begin{multline*}
x_0 + \dots + x_{n-1} = \sum_{k=0}^{n-1} \frac1{ R(2^{-k} v) }
 \log^{-(d-1)} S(2^{-k} v) \le \\
\le \frac{ \log^{-(d-1)} S(2\cdot 2^{-n} v) }{ R(2\cdot 2^{-n}v)
 } \sum_{k=0}^\infty 2^{-\frac k d} \, ;
\end{multline*}
we note that ultimately $ \log^{-(d-1)} S(2\cdot 2^{-n} v) \le 1 $
(since $ 2^{-n} v \to \infty $ by Lemma \ref{4.15}) and see that $ x_0
+ \dots + x_{n-1} \le \const(d) \cdot \( \frac{2^n}v \)^{1/d}
$. By \eqref{4.10} (the left-hand inequality),
\begin{multline*}
\frac{ \sqrt v }{ \eps |\la| } \Big( \frac{2^n}v \Big)^{1/d} <
 \frac{ \sqrt v }{ \eps |\la| } 2^{1/d} (2d)^{2d-2}  \frac{ ( C_3 |\la| )^2
 }{ v } \log^{2d-2} \frac{ \sqrt v }{ C_3 |\la| } \le \\
\le \const(d) \cdot \frac{C_3}\eps \cdot \frac{ C_3 |\la| }{ \sqrt v
 } \log^{2d-2} \frac{ \sqrt v }{ C_3 |\la| }
 \le \const(d) \cdot \frac{C_3}\eps \cdot \sqrt{ \frac{ C_3 |\la|
 }{ \sqrt v } }
\end{multline*}
when $ \frac{ \sqrt v }{ C_3 |\la| } $ is large enough (which holds
ultimately). By \eqref{***a},
$ { \frac{\eps^2}{C_3^2} \cdot \frac{ \sqrt v }{ C_3 |\la| } \to \infty
} $, whence $ \const(d) \cdot \frac{C_3}\eps \cdot \sqrt{ \frac{ C_3 |\la|
}{ \sqrt v } } \le 1 $ ultimately.
\end{proof}
\end{sloppypar}

Prop.~\ref{4.1} is thus proved.

Here is a similar lower bound.

\begin{proposition}\label{4.20}
For all $ \eps>0 $ and $ W \ge C_1 $ there exists $ C>0 $ such that for all boxes
$ B \subset \R^d $ and numbers $\la\in\R$ satisfying \eqref{4.2a},
\eqref{4.3a} there exist $n\in\{0,1,2,\dots\}$ and $\mu\in\R$ such
that
\begin{equation}
\width(B/2^n) \ge W \, , \quad
\frac1{\la^2} f_B(\la) \ge (1-\eps) \frac1{\mu^2} f_{B/2^n}(\mu) -
 \eps \, , \label{4.4'}
\end{equation}
and \eqref{4.5} holds.
\end{proposition}

\begin{proof}
As before, we deal with $C_2,C_3$ rather than $C$.
We apply Lemma \ref{2.3}(b) similarly to the proof of Prop.~\ref{4.1};
the logic is the same, but some formulas are a bit different:
\begin{equation}\tag{\ref*{4.7}$\,'$}\label{4.7'}
\frac{\sqrt2}{|\la_k|} - \frac1{|\la_{k+1}|} \le -y_k \quad \text{for
  all } k=0,\dots,n-1 \, ;
\end{equation}
\begin{equation}\tag{\ref*{4.8}$\,'$}\label{4.8'}
\frac1{ |\la_k| \sqrt{ 2^{-k} v } } f_{B/2^k} (\la_k) \ge \frac1{
|\la_{k+1}| \sqrt{ 2^{-(k+1)} v }} f_{B/2^{k+1}} (\la_{k+1}) - x_k \,
;
\end{equation}
$ 2^{n/2} |\mu| \ge (1-\eps) |\la| $ (instead of $ 2^{n/2} |\mu| \le
(1+\eps) |\la| $); \eqref{4.8'} gives $ \frac{ f_B(\la) }{ |\la|
  \sqrt v } \ge \frac{ f_{B/2^n}(\mu) }{ |\mu| \sqrt{2^{-n}v} } - \eps
\frac{ |\la| }{ \sqrt v } $, which implies the second inequality
of \eqref{4.4'}:
\[
\frac1{\la^2} f_B(\la) \ge \frac{ \sqrt v }{ |\la| } \Big( \frac{
f_{B/2^n}(\mu) }{ |\mu| \sqrt{2^{-n}v} } - \eps \frac{ |\la| }{
\sqrt v } \Big) = \frac{ 2^{n/2} |\mu| }{ |\la| } \cdot \frac{
f_{B/2^n}(\mu) }{ \mu^2 } - \eps \, .
\]
Still, $ x_k $ and $ y_k $ are the same as before; formulas
\eqref{4.6}, \eqref{4.9}, \eqref{4.10} and Lemmas \ref{4.11},
\ref{4.13} are intact. In contrast, $ \la_k $ are different:
\[
\frac1{ 2^{k/2} |\la_k| } = \frac1{|\la|} + \frac{C_1}{\sqrt v}
\sum_{i=0}^{k-1} \log^{d-1} S(2^{-i}v) \, ;
\]
for $ k = n $ we get
\begin{equation}\tag{\ref*{4.14}$\,'$}\label{4.14'}
2^{n/2} |\mu| \ge (1-\eps) |\la| \, ,
\end{equation}
since $ \frac{1+\eps}{1+2\eps} \ge 1-\eps $. In addition we note that
$ 2^{n/2} |\mu| \le |\la| $. Lemmas \ref{4.15}, \ref{4.17} are
intact. In the proof of Lemma \ref{4.16} the first line is slightly
different: $ (1+\eps) 2^{-n/2} |\la| \ge 2^{-n/2} |\la| \ge |\mu|
> \dots $
\end{proof}

%% file: sect5.tex
\emph{Third digression on boxes.}
Still, by a box we mean (unless stated otherwise) a set of the form $ B =
[0,r_1] \times \dots \times [0,r_d] \subset \R^d $ where $
r_1,\dots,r_d \in (0,\infty) $. Given a multiindex
$\al=(\al_1,\dots,\al_d)\in\{0,1,2,\dots\}^d $, we define
$|\al|=\al_1+\dots+\al_d $ and
\[
2^\al B = [0,2^{\al_1}r_1] \times \dots \times [0,2^{\al_d}r_d] \,
, \quad
2^{-\al} B = [0,2^{-\al_1}r_1] \times \dots \times
[0,2^{-\al_d}r_d] \, .
\]
Note that $ B/2^n $ is always of the form $ 2^{-\al}B $ for some
(evidently unique) $\al$, and $|\al|=n$.
Generally, $ (2^\al B)/2^{|\al|} \ne B $, but sometimes this holds.

\begin{lemma}\label{5.1}
If $ \length B < 2 \width B $, then $ (2^\al B)/2^{|\al|} = B $ for
every $\al$.
\end{lemma}

\begin{proof}
Lemma \ref{3.1} applied to the number $ C = \width B $ and the box $
2^\al B $ gives $n$ such that $ C \le \width\( (2^\al
B)/2^n \) \le \length\( (2^\al B)/2^n \) < 2C $. We know that $ (2^\al
B) / 2^n = 2^{-\be} (2^\al B) $ for some $\be$, and $ |\be|=n $. For
each $ k \in \{1,\dots,d\} $ we have $ r_k \in [C,2C) $ and $
2^{\al_k-\be_k} r_k \in [C,2C) $, which readily implies $ \be = \al $.
\end{proof}

\smallskip

\emph{Back to random fields.}
The constant $ C_1 $, the functions $ R(\cdot), S(\cdot) $ and the
random field $X$ are as before.
Here is some convergence of functions $ f_B $ to a quadratic
function, as $ \width B \to \infty $.

\begin{proposition}\label{pr}
There exist $ \si \ge 0 $ and $ C>0 $ such that for every $ \eps>0 $
there exists $ W>0 $ such that the following inequality holds for all
boxes $B$ such that $ \width B \ge W $ and all $\la$ such that $
C|\la| \le \sqrt{S(\vol B)} \log^{-(d-1)} S(\vol B) $:
\[
\bigg| f_B(\la) - \frac12 \si^2 \la^2 \bigg| \le \eps \la^2 \, .
\]
\end{proposition}

Similarly to \cite[Sect.~3]{II} (and dissimilarly to \eqref{4.2}) we
denote by $ C_2 $ the constant given by \cite[Prop.~2.3]{II} (it
depends on the random field $X$ only); $ C_2 \ge 1 $.
We'll prove that the claim of Prop.~\ref{pr} (``for every $\eps>0$
there exists\dots'') holds \emph{ultimately}
in the following sense: it admits a sufficient condition of the
form\footnote{%
 This is the meaning of the word ``ultimately'' throughout this
 section (and only in this section).}
\[
C \ge C_0 (d,C_1,C_2) \, .
\]

\begin{lemma}\label{5.12}
Ultimately, if $ \vol B \ge C $ and $ \width B \ge C_2 $, then $
f_B(\la) \le C_2 \la^2 $ for all $ \la \in [-\frac1{C_2}, \frac1{C_2}]
$; in particular, $ f_B \( \pm\frac1{C_2} \) \le \frac1{C_2} \le 1 $. 
\end{lemma}

\begin{proof}
By \cite[(3.1)]{II}, $ f_B(\la) \le C_2 \la^2 $ whenever $ C_2 |\la|
\le \frac{ \sqrt{S(v)} }{ \log^{d-1} S(v) } $ and $ \width B \ge C_2 $;
here $ v = \vol B $. Ultimately,
$ \inf_{v\in[C,\infty)} \sqrt{S(v)} \log^{-(d-1)} S(v) \ge 1 $.
\end{proof}

\emph{Remark.} For $d>1$ we may do better, taking $C$ such that the
infimum exceeds (say) $ \sqrt{C_2} $. But for $d=1$ this infimum is
always $1$.

\begin{lemma}\label{5.11}
Ultimately, for all numbers $ \de>0 $ and boxes $ B
\subset \R^d $ satisfying $ 2C_2\de \le 1 $, $ \vol B \ge C $ and $
\width B \ge C_2 $ holds
\[
\sup_{0<|\la|\le\de} \frac1{\la^2} f_B(\la) - \inf_{0<|\la|\le\de}
\frac1{\la^2} f_B(\la) \le C \de \, .
\]
\end{lemma}

\begin{proof}
We apply \cite[Lemma 2d1]{I} to the random variable
$ \frac1{C_2\sqrt{\vol B}} \int_B X_t \, \D t $ and its cumulant
generating function $ f(\la) = \log \Ex \exp \frac{\la}{C_2\sqrt{\vol
B}} \int_B X_t \, \D t = f_B \( \frac{\la}{C_2} \) $.
Taking into account that $ f(\pm1) \le 1 $ by Lemma \ref{5.12}, we get
\[
\Big| \frac1{\la^2} f(C_2\la) - \frac{C_2^2}2 f''(0) \Big| \le
\frac{41}{6\E^3} \frac1{\la^2} \Big( \frac{ C_2|\la| }{ 1 - C_2|\la| }
\Big)^3 \cdot 2\E \quad \text{for } 0 < C_2 |\la| < 1 \, .
\]
Thus,
\[
\sup_{0<|\la|\le\de} \frac1{\la^2} f_B(\la) - \inf_{0<|\la|\le\de}
\frac1{\la^2} f_B(\la) \le \frac{82}{3\E^2} \frac{C_2^3 \de}{(1-C_2\de)^3}
\]
for $ C_2\de < 1 $; and for $ 2C_2\de \le 1 $ the right-hand side does not
exceed $ \frac{82}{3\E^2} (2C_2)^3 \de $. Ultimately,
$ C \ge \frac{82}{3\E^2} (2C_2)^3 $.
\end{proof}

We consider the quadratic bounds (lower and upper)
\begin{equation}\label{qb}
L(B) \la^2 \le f_B(\la) \le U(B) \la^2 \quad \text{for }
|\la| \le \De(\vol B) \, ,
\end{equation}
where $ \De(v) = \frac1C \sqrt{S(v)} \log^{-(d-1)} S(v) $ and
\[
L(B) = \inf_{0<|\la|\le \De(\vol B)} \frac1{\la^2} f_B(\la) \, , \quad
U(B) = \sup_{0<|\la|\le \De(\vol B)} \frac1{\la^2} f_B(\la) \, .
\]

\begin{lemma}\label{5.4}
Ultimately, for every $ \eps>0 $ there exist $W$ and $N$ such that the
inequality
\[
\sup_{|\al|\ge N} U(2^\al B) - \inf_{|\al|\ge N} L(2^\al
B) \le \eps
\]
holds for all boxes $B$ satisfying $ \length B < 2\width B $ and $
W \le \width B \le \length B \le \frac{W}{\eps^{1/d}} $.
\end{lemma}

\begin{proof}
We'll prove rather that $ \sup\dots - \inf\dots \le 3\eps $.
Without loss of generality we assume that $ \eps \le 2^{-d} $ and
$ \eps \le \frac1{C_2} $. 
We denote by $ V_\eps^{\text{\ref*{3.3}}} $ the constant $V_\eps$
implicit in Prop.~\ref{3.3} (and explicit in Prop.~\ref{eight}); $
V_\eps^{\text{\ref*{3.11}}} $ is the $V_\eps$ implicit in
Prop.~\ref{3.11}.

We take $W$ such that $ W \ge C_2 $ and the number $ V = W^d $
satisfies $ V \ge C $, $ V \ge V_\eps^{\text{\ref*{3.3}}} $ and $
V \ge V_\eps^{\text{\ref*{3.11}}} $.

We take $ \de = \frac \eps C $.
For $\eps$, $V$ and $\de$ Prop.~\ref{eight} gives  a natural $N$,
denote it $ N_{\text{\ref*{3.3}}} $; similarly, Prop.~\ref{3.11} gives
$ N_{\text{\ref*{3.11}}} $; we take $ N = \max( N_{\text{\ref*{3.3}}},
N_{\text{\ref*{3.11}}} ) $.

We introduce $ a_{\min} = \inf_{0<|\la|\le\de} \frac1{\la^2} f_B(\la)
$, $ a_{\max} = \sup_{0<|\la|\le\de} \frac1{\la^2} f_B(\la)
$. Lemma \ref{5.12} applies (since $ \width B \ge W \ge C_2 $ and
$ \vol B \ge W^d = V \ge C $),
giving $ a_{\max} \le C_2 $ (since $ \de \le \eps \le \frac1{C_2} $).

Ultimately $ C \ge 2 C_2 $. Lemma~\ref{5.11} applies to $\de$ and $B$
(since $ 2 C_2 \de \le C \de = \eps \le C_2 \eps \le 1 $, $ \vol B \ge
C $ and $ \width B \ge C_2 $), giving $ a_{\max} - a_{\min} \le \eps
$. It is sufficient to prove for all $\al$ satisfying $ |\al| \ge N $
two inequalities
\[
L(2^\al B) \ge a_{\min} - \eps \, ,
\quad 
U(2^\al B) \le a_{\max} + \eps \, .
\]
Denoting $ n = |\al| $ we have $ (2^\al B)/2^n = B $ by
Lemma \ref{5.1} (since $ \length B < 2 \width B $). Prop.~\ref{eight}
applies to $ \eps $, $V$, $\de$, $n$, $ a_{\text{\ref*{3.3}}} =
a_{\max} $ and $ B_{\text{\ref*{3.3}}} = 2^\al B $ (since $ V \ge
V_\eps^{\text{\ref*{3.3}}} $, $ n = |\al| \ge N \ge
N_{\text{\ref*{3.3}}} $, $ a_{\max} \le C_2 \le \frac1{\eps} $, $ 2^n
V \le \vol(2^\al B) = 2^n \vol B \le 2^n (\length B)^d \le \frac1\eps
\cdot 2^n V $, and $ \width (2^\al B) \ge \width B \ge C_2 \ge C_1 $),
giving $ U(2^\al B) \le a_{\max} + \eps $ (since $
C_1 \sqrt{a_{\max}+\eps} \le C_1 \sqrt{ C_2 + \frac1{C_2} } \le C $
ultimately). Similarly, Prop.~\ref{3.11} gives $ L(2^\al B) \ge
a_{\min} - \eps $.
\end{proof}

\begin{lemma}\label{5.6}
Ultimately, for every box $B$ the two limits
$ \lim_{\al_1,\dots,\al_d\to\infty} L(2^\al B) $,
$ \lim_{\al_1,\dots,\al_d\to\infty} U(2^\al B) $ exist and are
equal. (Here $ \al = (\al_1,\dots,\al_d) $.)
\end{lemma}

\begin{sloppypar}
\begin{proof}
Clearly, for every box $B$ and multiindex $ \be \in \{0,1,\dots\}^d $,
\[
\limsup_{\al_1,\dots,\al_d \to \infty} U(2^\al B)
= \limsup_{\al_1,\dots,\al_d \to \infty} U(2^\al 2^\be B)
\]
(here $ \limsup_{\al_1,\dots,\al_d \to \infty} \dots $ means
$ \lim_{n\to\infty} \sup_{\al_1,\dots,\al_d \ge n} \dots $), and
similarly $ \liminf_{\al_1,\dots,\al_d \to \infty} L(2^\al B)
= \liminf_{\al_1,\dots,\al_d \to \infty} L(2^\al 2^\be B) $. It is
sufficient to prove that $ \limsup_{\al_1,\dots,\al_d \to \infty}
U(2^\al B) - \liminf_{\al_1,\dots,\al_d \to \infty} L(2^\al
B) \le \eps $ for all $ \eps \in (0,2^{-d}] $. For $\eps$,
Lemma \ref{5.4} gives $W$ and $N$. WLOG, $ \width B \ge W $ (otherwise
use $ 2^\be B $ for appropriate $\be$).
WLOG, $ \length B < 2W $ (otherwise use $ B/2^n $ given by
Lemma \ref{3.1}). Clearly,
$ \limsup_{\al_1,\dots,\al_d \to \infty} U(2^\al
B) \le \sup_{|\al|\ge N} U(2^\al B) $, and similarly,
$ \liminf\dots \ge \inf\dots $
Lemma \ref{5.4} applies to $\eps$ and $B$ (since $ \length B <
2W \le \frac{W}{\eps^{1/d}} $), and we get
\[
\limsup_{\al_1,\dots,\al_d \to \infty} U(2^\al
B) - \liminf_{\al_1,\dots,\al_d \to \infty} L(2^\al B)
\le \sup_{|\al|\ge N} U(2^\al B) - \inf_{|\al|\ge N} L(2^\al
B) \le \eps \, . 
\vspace{-15pt}
\]
\end{proof}
\end{sloppypar}

\begin{remark}\label{5.7}
The convergence in Lemma \ref{5.6} is uniform over all boxes $B$
satisfying $ \width B \ge 1 $.

Proof. The only danger to uniformity (in the proof of Lemma \ref{5.6})
is the transition from $B$ to $2^\be B$ satisfying $ \width(2^\be
B) \ge W $. We take $ m \in \{0,1,\dots\} $ such that $ 2^m \ge W $,
and let $ \be = (m,\dots,m) $, then $ \width B \ge 1 $ implies
$ \width(2^\be B) \ge 2^m \ge W $.
\end{remark}

Ultimately, for every box $B$ we define $ \si_B \in [0,\infty) $ by
\[
\frac12 \si_B^2 = \lim_{\al_1,\dots,\al_d\to\infty} L(2^\al B) =
\lim_{\al_1,\dots,\al_d\to\infty} U(2^\al B) \, .
\]
Clearly,
\[
\si_{2^\al B} = \si_B \quad \text{for all } \al \text{ and } B \, .
\]

\begin{lemma}\label{5.8}
Ultimately, the function $ (r_1,\dots,r_d) \mapsto
\si_{[0,r_1]\times\dots\times[0,r_d]} $ from $ (0,\infty)^d $ to $
[0,\infty) $ is Borel measurable.
\end{lemma}

\begin{proof}
First, for each $ \la \in \R $, the function $ (r_1,\dots,r_d) \mapsto
f_{[0,r_1]\times[0,r_d]} (\la) $ from $ (0,\infty)^d $ to $ [0,\infty]
$ is Borel measurable due to \cite[(1.3)]{II}. Second, for each $\la$
such that $ |\la| \le \frac1C $ we have\footnote{%
 We have it for all $ \la \in \R $ when $ d>2 $, but only for $
 |\la| \le \frac1C $ when $ d=1 $.}
\begin{equation}\label{converg}
f_{[0,2^n r_1]\times[0,2^n r_d]}
(\la) \to \frac12 \si^2_{[0,r_1]\times[0,r_d]} \la^2 \quad \text{as }
n \to \infty
\end{equation}
by Lemma \ref{5.6} and \eqref{qb}. It remains to apply this to a
single $ \la $ such that $ 0 < |\la| \le \frac1C $, and divide by
$\la^2 $.
\end{proof}

\begin{lemma}
Ultimately, $ \si_B $ does not depend on $B$.
\end{lemma}

\begin{proof}
We'll prove that $ \si_{[0,r_1]\times\dots\times[0,r_d]} $ does not
depend on $ r_1 $; the same argument works for each $ r_k $. Given $
r_2,\dots,r_d > 0 $, we denote $ B_r = [0,r] \times
[0,r_2] \times \dots \times [0,r_d] $ for $ r>0 $, and $ \si_r
= \si_{B_r} $.

We'll prove that the function $ r \mapsto r\si^2_r $ is
linear.\footnote{%
 Similar to \cite[3b1]{I}.}
Its measurability being ensured by Lemma \ref{5.8}, we prove
additivity:
\begin{equation}\label{addi}
(r+s) \si^2_{r+s} = r \si^2_r + s \si^2_s \, .
\end{equation}
We use \cite[Lemma 2.10]{II} again:\footnote{%
 See the proof of Lemma \ref{2.1}(b).}
for all $ p \in (1,\infty) $,
\begin{multline*}
p f_{B_r} \bigg( \frac{\la}{p} \sqrt{ \frac{r}{r+s} } \bigg) +
 p f_{B_s} \bigg( \frac{\la}{p} \sqrt{ \frac{s}{r+s} } \bigg) -
 (p-1) g_{B_0} \bigg( \frac1{p-1} \frac{\la}{\sqrt{r+s}} \bigg) \le \\
\le f_{B_{r+s}}(\la) \le 
 \frac1p f_{B_r} \bigg( p\la \sqrt{ \frac{r}{r+s} } \bigg) +
 \frac1p f_{B_s} \bigg( p\la \sqrt{ \frac{s}{r+s} } \bigg) +
 \frac{p-1}p
 g_{B_0} \bigg( \frac{p}{p-1} \frac{-\la}{\sqrt{r+s}} \bigg) \, ,
\end{multline*}
where $ g_{B_0}(\mu) \le C_1 \mu^2 $ for $
C_1 |\mu| \le \frac{\sqrt{\vol B_0}}{\log^{d-1} \vol B_0 } $ and
$ \vol B_0 = \frac{ \vol B_{r+s} }{r+s} = \frac{ \vol B_r }{r}
= \frac{ \vol B_s }{s} $ (provided that $ \width B_0 \ge C_1 $; but
see below).

The same applies to rescaled boxes $ 2^n B_r $ (that is, $ 2^\al B_r $
for $ \al = (n,\dots,n) $), $ 2^n B_s \subset \R^d $ and $ 2^n
B_0 \subset \R^{d-1} $, for all $n$. Instead of $
f_{B_r} \( \frac{\la}{p} \sqrt{ \frac{r}{r+s} } \) $ we have $
f_{2^n B_r} \( \frac{\la}{p} \sqrt{ \frac{2^n r}{2^n r + 2^n s}
} \) = f_{2^n B_r} \( \frac{\la}{p} \sqrt{ \frac{r}{r+s} } \) $,
which turns into $ \frac12 \si^2_r \frac{\la^2}{p^2} \frac{r}{r+s}
$ in the limit $ n \to \infty $, for $ |\la| \le \frac1C $,
by \eqref{converg}. The other four terms with $f$ are treated
similarly. But the terms with $g$ are dissimilar; instead of $
g_{B_0} \( \frac1{p-1} \frac{\la}{\sqrt{r+s}} \) $ we have $
g_{2^n B_0} \( \frac1{p-1} \frac{\la}{\sqrt{2^n r + 2^n s}} \)
$, which vanishes in the limit $ n \to \infty $, since $ g_{2^n
B_0}(\mu) \le C_1 \mu^2 $ (provided that $ \width(2^n B_0) \ge C_1 $,
which holds for large $n$) for $ |\mu| \le \frac1{C_1} $.\footnote{%
 And moreover, for $ C_1 |\mu| \le \frac{\sqrt{2^{n(d-1)}\vol
 B_0}}{\log^{d-1} (2^{n(d-1)}\vol B_0 ) } $.}
Thus, we get in the limit, after multiplication by $ \frac2{\la^2} $,
\[
\frac1p \cdot \frac{r}{r+s} \si^2_r
+ \frac1p \cdot \frac{s}{r+s} \si^2_s \le \si^2_{r+s} \le
p \cdot \frac{r}{r+s} \si^2_r + p \cdot \frac{s}{r+s} \si^2_s \, .
\]
Now the limit for $ p \to 1+ $ gives the additivity:
$ \frac{r}{r+s} \si^2_r + \frac{s}{r+s} \si^2_s \le \si^2_{r+s} \le
\frac{r}{r+s} \si^2_r + \frac{s}{r+s} \si^2_s $.
\end{proof}

Ultimately, we define $ \si \in [0,\infty) $ by
\[
\si = \si_B \quad \text{for all boxes } B \subset \R^d \, .
\]

\begin{proof}[Proof of Prop.~\ref{pr}]
Ultimately, given $ \eps>0 $, Remark \ref{5.7} gives $N$ such that the
inequality
\begin{equation}\label{5.13}
\frac12 \si^2 - \eps \le L(2^\al B_1) \le U(2^\al
B_1) \le \frac12 \si^2 + \eps
\end{equation}
holds for all $ \al \in \{N,N+1,\dots\}^d $ and all boxes $B_1$ such
that $ \width B_1 \ge 1 $. We take $ W = 2^N $. Given a box $B$ such
that $ \width B \ge W $, we have $ B = 2^N B_1 $ with $ B_1 $ 
such that $ \width B_1 \ge 1 $. Now the needed inequality
follows from \eqref{5.13} and \eqref{qb}.
\end{proof}

%% file: sect6.tex
\begin{proposition}\label{6.1}
There exists $ \si \ge 0 $ such that for every $ \eps > 0 $ there exist
$ C > 0 $ and $ W > 0 $ such that the following inequality holds for
all boxes $ B $ such that $ \width B \ge W $ and all $ \la $ such that
$ C |\la| \le \sqrt{ \vol B } \log^{-d} \vol B $:
\[
\Big| f_B(\la) - \frac12 \si^2 \la^2 \Big| \le \eps \la^2 \, .
\]
\end{proposition}

\begin{proof}
Prop.~\ref{pr} gives $ \si $ and $ C_{\text{\ref*{pr}}} $.

We take $ \eps_1>0 $ such that $ (1+\eps_1) \( \frac12 \si^2 + \eps_1
\) + \eps_1 \le \frac12 \si^2 + \eps $ and $ C_{\text{\ref*{pr}}}
\eps_1 \le 1 $.

For $ \eps_1 $, Prop.~\ref{pr} gives $ W_{\text{\ref*{pr}}} $; WLOG, $
W_{\text{\ref*{pr}}} \ge C_1 $.

For $ \eps_1 $ and $ W_{\text{\ref*{pr}}} $, Prop.~\ref{4.1} gives
$C_{\text{\ref*{4.1}}}$.

We take $ C = C_{\text{\ref*{4.1}}} $ and $ W = \max \(
W_{\text{\ref*{pr}}}, C^{1/d} \) $.

For $ \la=0 $ the claimed inequality is trivial. For other $\la$,
Prop.~\ref{4.1} applies to $\eps_1$, $ W_{\text{\ref*{pr}}} $, $B$ and
$\la$ (since $ \vol B \ge (\width B)^d \ge W^d \ge C =
C_{\text{\ref*{4.1}}} $, and $ \width B \ge W \ge W_{\text{\ref*{pr}}} $, and
$ 0 < C_{\text{\ref*{4.1}}} |\la| = C |\la| \le \sqrt{ \vol B }
\log^{-d} \vol B $), giving $ n $ and $ \mu $ that satisfy
\eqref{4.4}, \eqref{4.5}.

Prop.~\ref{pr} applies to $ \eps_1 $, $ B/2^n $ and $ \mu $ (since $
\width(B/2^n) \ge W_{\text{\ref*{pr}}} $ by \eqref{4.4} and, using
\eqref{4.5}, $ C_{\text{\ref*{pr}}} |\mu| \le \frac1{\eps_1} |\mu| \le
\sqrt{ S(\vol(B/2^n)) } \log^{-(d-1)} \vol(B/2^n) \le \sqrt{
S(\vol(B/2^n)) } \log^{-(d-1)} S(\vol(B/2^n)) $),
giving $ \big| f_{B/2^n}(\mu) - \frac12 \si^2 \mu^2
\big| \le \eps_1 \mu^2 $. Combined with \eqref{4.4} it gives the upper
bound:
\[
\frac1{\la^2} f_B(\la) \le (1+\eps_1) \Big( \frac12 \si^2 + \eps_1
\Big) + \eps_1 \le \frac12 \si^2 + \eps \, .
\]
Similarly, using Prop.~\ref{4.20} instead of \ref{4.1}, we get the
lower bound.
\end{proof}

\begin{proof}[Proof of Theorem \ref{theorem1}]
In terms of the box $ B = [0,r_1] \times \dots \times [0,r_d] $ we
have
\begin{multline*}
\frac1{ r_1\dots r_d\la^2 } \log \Ex \exp \la
 \int_{[0,r_1]\times\dots\times[0,r_d]} X_t \, \D t = \\
= \frac1{\la^2 \vol B} \log \Ex \exp \la \int_B X_t \, \D t =
 \frac1{\la^2 \vol B} f_B \( \la \sqrt{\vol B} \) \, .
\end{multline*}
Prop.~\ref{6.1} gives $\si$. 
Given $ \eps>0 $, we need $ R $ and $ \de $ such that
\[
\Big| \frac1{\la^2 \vol B} f_B \( \la \sqrt{\vol B} \) - \frac12 \si^2
\Big| \le \eps
\]
whenever $ \width B \ge R $ and $ 0 < |\la| \le \de \log^{-d} \vol B
$; that is, $ \big| \frac1{\la^2} f_B (\la) - \frac12 \si^2 \big| \le
\eps $ for $  0 < |\la| \le \de \sqrt{\vol B} \log^{-d} \vol B $, or
\[
\Big| f_B(\la) - \frac12 \si^2 \la^2 \Big| \le \eps \la^2 \quad
\text{for } |\la| \le \de \sqrt{\vol B} \log^{-d} \vol B \, .
\]
For $\eps$, Prop.~\ref{6.1} gives $C$ and $W$. It remains to take $ R
= W $ and $ \de = 1/C $.
\end{proof}

Corollaries \ref{1.2}, \ref{1.3} follow from Theorem \ref{theorem1} in
the same way as \cite[1.7, 1.8]{I} follow from \cite[1.6]{I}. The
multidimensional convergence reduces readily to convergence of
sequences.

The implication \ref{theorem1} \imp \ref{1.2} being a special case of
the G\"artner-Ellis theorem, I provide detailed calculations for the
relevant special case of this general theorem in a lemma about just
probability distributions on $\R$ (rather than random fields on
$\R^d$).

\begin{lemma}\label{6.2}
Let $ f : \R \to (-\infty,+\infty] $ be the cumulant generating
function of a probability measure $\mu$ on $\R$; that is,
\[
f(\la) = \log \int_{\R} \E^{\la x} \, \mu(\D x) \quad \text{for } \la
\in \R \, .
\]
If numbers $ \eps \in \(0,\frac1{100}] $ and $ A > a \ge 100 $ satisfy
\[
\Big( \frac12 - \eps \Big) \la^2 \le f(\la) \le \Big( \frac12 + \eps
\Big) \la^2 \quad \text{for all } \la \in [a,A+6(1+A\sqrt\eps)] \, ,
\]
then
\[
0.96 \exp \bigg( \! - \Big( \frac12 + 16\eps + \frac{9}{x} \Big) x^2
\bigg) \le \mu \( (x,\infty) \) \le \mu \( [x,\infty) \) \le \exp
\bigg( \! - \Big( \frac12 - \eps \Big) x^2 \bigg)
\]
for all $ x \in [a,A] $.
\end{lemma}

\begin{proof}
\begin{sloppypar}
For the upper bound, the proof is simple: $ \mu \( [x,\infty) \) \le
\E^{-x^2} \int_{(x,\infty)} \E^{xy} \, \mu(\D y) \le \exp \( -x^2 +
(0.5+\eps) x^2 \) $. We turn to the lower bound.
\end{sloppypar}

For all $ x,\de,\la \ge 0 $ we have
\begin{gather*}
\int_{(-\infty,x]} \E^{\la y} \, \mu(\D y) \le \E^{\de x}
 \int_{(-\infty,x]} \E^{(\la-\de) y} \, \mu(\D y) \, , \\
\int_{[x+2\de,\infty)} \E^{\la y} \, \mu(\D y) \le \E^{-\de(x+2\de)}
 \int_{[x+2\de,\infty)} \E^{(\la+\de) y} \, \mu(\D y) \, .
\end{gather*}
Assuming that $ a \le \la-\de \le \la+\de \le A + 6(1+A\sqrt\eps) $ we get
\begin{gather*}
\int_{(-\infty,x]} \E^{\la y} \, \mu(\D y) \le \exp \( \de x +
 (0.5+\eps) (\la-\de)^2 \) \, , \\
\int_{[x+2\de,\infty)} \E^{\la y} \, \mu(\D y) \le \exp \(
 -\de(x+2\de) + (0.5+\eps) (\la+\de)^2 \) \, .
\end{gather*}
From now on we take
\[
\la = x + \de \, .
\]
We note that
\begin{multline*}
\( \de x + (0.5+\eps) (\la-\de)^2 \) - \( -\de(x+2\de) + (0.5+\eps)
 (\la+\de)^2 \) = \\
= 2\de x + 2\de^2 - (0.5+\eps) \cdot 4\la\de = 2\de \( x+\de -
(0.5+\eps) (2x+2\de) \) = -4\eps\de(x+\de) \le 0 \, ,
\end{multline*}
therefore, assuming that $ 2\de \le 6(1+A\sqrt\eps) $,
\begin{multline*}
\int_{(x,x+2\de)} \E^{\la y} \, \mu(\D y) \ge \int_{\R} \E^{\la y} \,
 \mu(\D y) - 2 \exp \( -\de(x+2\de) + (0.5+\eps) (\la+\de)^2 \) \ge \\
\ge \exp \( (0.5-\eps) \la^2 \) - 2 \exp ( \dots ) =
\exp \( (0.5-\eps) \la^2 \) ( 1 - 2 \exp \xi ) \, ,
\end{multline*}
where
\[
\xi = -\de(x+2\de) + (0.5+\eps) (\la+\de)^2 - (0.5-\eps) \la^2 \, .
\]
Still, $ \la = x+\de $, thus
\begin{multline*}
\xi = -\de(x+2\de) + (0.5+\eps) (x+2\de)^2 - (0.5-\eps) (x+\de)^2 = \\
= 2\eps x^2 + 6\eps\de x + 5\eps\de^2 - 0.5 \de^2 \, .
\end{multline*}
We choose $\de$ as follows:
\[
\de = 3(1+x\sqrt\eps) \, ;
\]
now
\begin{multline*}
\xi = 2\eps x^2 + 6\eps x \cdot 3(1+x\sqrt\eps) + (5\eps-0.5) \cdot
 9(1+x\sqrt\eps)^2 = \\
= -9 (0.5-5\eps) - x^2 \eps \Big( 2.5 - 45\eps - 18\sqrt\eps -
 \frac{18}x \Big) - 18 x \sqrt\eps (0.5-5\eps) \, .
\end{multline*}
For $ \eps \le 0.01 $ and $ x \ge 100 $ we have $ 2.5 - 45\eps -
18\sqrt\eps - \frac{18}x \ge 2.5 - 0.45 - 1.8 - 0.18 = 0.07 \ge 0 $,
thus, $ \xi \le -9 \cdot 0.45 = -4.05 $; $ 1 - 2 \exp \xi \ge 1 - 2
\cdot 0.018 \ge 0.96 $;
\[
\int_{(x,x+2\de)} \E^{\la y} \, \mu(\D y) \ge 0.96 \exp \( (0.5-\eps)
\la^2 \) \, .
\]
Therefore
\begin{multline*}
\mu \( (x,\infty) \) \ge \mu \( (x,x+2\de) \) \ge \E^{-\la(x+2\de)}
 \int_{(x,x+2\de)} \E^{\la y} \, \mu(\D y) \ge \\
\ge 0.96 \exp \( (0.5-\eps)
 \la^2 - \la (x+2\de) \) = 0.96 \exp ( -0.5 x^2 - \zeta ) \, ,
\end{multline*}
where
\begin{multline*}
\zeta = -0.5 x^2 - (0.5-\eps) \la^2 + \la (x+2\de) = \\
 = -0.5 x^2 - (0.5-\eps) (x+\de)^2 + (x+\de)(x+2\de) = \\
 = \eps x^2 + (2+2\eps) x \de + (1.5+\eps) \de^2 = \\
 = \eps x^2 + 2(1+\eps)x \cdot 3(1+x\sqrt\eps) + (1.5+\eps) \cdot
  9(1+2x\sqrt\eps+x^2 \eps) = \\
 = \eps ( 14.5 + 6\sqrt\eps + 9\eps ) x^2 +
 ( 6 + 27\sqrt\eps + 6\eps + 18\eps\sqrt\eps ) x + ( 13.5 + 9\eps )
 \, ; 
\end{multline*}
for $ \eps \le 0.01 $ and $ x \ge 100 $ we have $ \zeta \le 15.19 \eps
x^2 + 8.778 x + 13.59 \le \( 15.19 \eps + \frac{8.9139}{x} \) x^2 \le \(
16\eps + \frac9x \) x^2 $.
\end{proof}

\begin{proof}[Proof of Corollary \ref{1.2}]
In terms of the box $ B = [0,r_1] \times \dots \times [0,r_d] $ we
have
\[
\PR{ \int_{[0,r_1]\times\dots\times[0,r_d]} X_t \, \D t \ge c\si
\sqrt{r_1\dots r_d} } = \PR{ \int_B X_t \, \D t \ge c\si \sqrt{\vol
B} } \, .
\]
Given $\eps$, we need $ R $ and $ \de $ such that
\[
\bigg| \frac1{c^2} \log \PR{ \int_B X_t \, \D t \ge c\si \sqrt{\vol B}
} + \frac12 \bigg| \le \eps
\]
whenever $ \width B \ge R $ and $ R \le c \le \sqrt{ \de
\vol B } \log^{-d} \vol B $. WLOG, $ \eps \le 0.32 $.

For $ \eps_{\text{\ref*{6.1}}} = \frac{ \si^2 }{ 32 } \eps $,
Prop.~\ref{6.1} gives $ C_{\text{\ref*{6.1}}} $ and $ W $.
We take $ R = \max(W,100,\frac{10}\eps) $ and $ \de = \( \frac{ \si }{
1.66 C_{\text{\ref*{6.1}}} } \)^2 $.

Prop.~\ref{6.1} applies to $ \eps_{\text{\ref*{6.1}}} $ and $ B $
(since $ \width B \ge R \ge W $), giving $ | f_B(\la) - 0.5 \si^2 \la^2
| \le \eps_{\text{\ref*{6.1}}} \la^2 $ for $ C_{\text{\ref*{6.1}}}
|\la| \le \sqrt{v} \log^{-d} v $, where $ v = \vol B $.

We consider the distribution $ \mu $ of the random variable $
\frac1{\si\sqrt v} \int_B X_t \, \D t $ and its cumulant generating
function $f$; $ f(\la) = \log \Ex \exp \frac{\la}{\si\sqrt v} \int_B
X_t \, \D t = f_B \( \frac\la\si \) $. In order to apply Lemma
\ref{6.2} we take $ \eps_{\text{\ref*{6.2}}} = \frac1{32} \eps =
\frac{ \eps_{\text{\ref*{6.1}}} }{ \si^2 } \in (0,0.01] $, $ a=100 $
and $ A = \sqrt{ \de v } \log^{-d} v $. We note that $ a = 100 \le R
\le c \le A $ and $ A + 6(1+A\sqrt{\eps_{\text{\ref*{6.2}}}}) =
A(1+6\sqrt{\eps_{\text{\ref*{6.2}}}}) + 6 \le A(1+6\cdot 0.1) + 6 \le
1.66 A $. Lemma \ref{6.2} applies, since for all $ \la \in
[a,A+6(1+A\sqrt{\eps_{\text{\ref*{6.2}}}})] $ we have $
C_{\text{\ref*{6.1}}} \big| \frac\la\si \big| \le \frac{ 1.66A }{
C_{\text{\ref*{6.1}}} \si } = \frac{ 1.66 }{ C_{\text{\ref*{6.1}}} \si
} \sqrt{\de v} \log^{-d} v = \sqrt{v} \log^{-d} v $, whence $ |
f(\la) - 0.5 \la^2 | = \big| f_B \( \frac\la\si \) - 0.5 \si^2 \(
\frac\la\si \)^2 \big| \le \eps_{\text{\ref*{6.1}}} \(
\frac\la\si \)^2 = \eps_{\text{\ref*{6.2}}} \la^2 $.

\begin{sloppypar}
Clearly, $ \PR{ \int_B X_t \, \D t \ge c\si \sqrt{\vol B} } =
\mu\([c,\infty)\) $. For all $ c \in [R,\sqrt{\de v} \log^{-d} v ]
\subset [a,A] $ Lemma \ref{6.2} gives $ | \log \mu\([c,\infty)\) + 0.5
c^2 | \le |\log 0.96| + ( 16 \eps_{\text{\ref*{6.2}}} + \frac9c ) c^2
$; it remains to check that $ |\log 0.96| + ( 16
\eps_{\text{\ref*{6.2}}} + \frac9c ) c^2 \le \eps c^2 $, that is, $
\frac{|\log 0.96|}{\eps c^2} + \frac{0.5}{c^2} + \frac9{\eps c} \le 1
$. Taking into account that $ c \ge R \ge \frac{10}\eps $, that is, $
\eps c \ge 10 $, we have $ \frac{|\log 0.96|}{\eps c^2} +
\frac{0.5}{c^2} + \frac9{\eps c} \le \frac{0.05}{100\eps c} +
\frac{0.5}{100^2} + \frac9{\eps c} \le \frac{0.05}{100\cdot 10} +
\frac{0.5}{100^2} + \frac9{10} = 0.9001 \le 1 $.
\end{sloppypar}
\vspace{-12pt}
\end{proof}